\documentclass[a4paper]{article}

\usepackage{amsfonts}
\usepackage{graphicx}
\usepackage{amsthm}
\usepackage{amsmath}
\usepackage{amssymb}
\usepackage{enumerate}
\allowdisplaybreaks

\newtheorem{proposition}{Proposition}[section]
\newtheorem{lemma}[proposition]{Lemma}
\newtheorem{corollary}[proposition]{Corollary}

\newtheorem{definition}[proposition]{Definition}
\theoremstyle{definition}
\newtheorem{example}[proposition]{Example}
\numberwithin{equation}{section}

\begin{document}

\begin{center}
\LARGE
\textbf{Notes on well-distributed minimal sub-BIBDs for $\lambda=1$}
\bigskip\bigskip

\large
Daniele Dona\footnote{The author was partially supported by the European Research Council under Programme H2020-EU.1.1., ERC Grant ID: 648329 (codename GRANT).}
\bigskip

\normalsize
Mathematisches Institut, Georg-August-Universit\"at G\"ottingen

Bunsenstra\ss e 3-5, 37073 G\"ottingen, Germany

\texttt{daniele.dona@mathematik.uni-goettingen.de}
\bigskip\bigskip\bigskip
\end{center}

\begin{minipage}{110mm}
\small
\textbf{Abstract.} In these notes we investigate BIBDs with $\lambda=1$ that present subdesigns evenly covering both blocks and vertices: we determine some of their basic properties, consequence of already existing results in the literature, with regards to their size and the number of intersections of pairs and triples of subdesigns of a specific kind. We also describe the link between these particular BIBDs and the graph isomorphism problem, based on Babai's paper \cite{Ba15}, and point out the characteristics of these designs that would lead to improvements of the algorithm for the GIP.
\medskip

\medskip

\end{minipage}
\bigskip

\section{Introduction}

An important branch of combinatorics is \textit{design theory}, which investigates pairs $(V,\mathcal{B})$ with $V$ a finite set and $\mathcal{B}$ a collection of subsets of $V$ satisfying some generic ``nice'' properties with regards to inclusions and intersections: the umbrella term usually employed to describe such pairs is \textit{combinatorial designs}, or simply \textit{designs}. The aforementioned properties can be of many different kinds, giving birth to various designs, of interest for different reasons in disparate fields, such as statistics and geometry; for dozens of examples of designs we refer the reader to Stinson \cite{St04} and Colbourn-Dinitz \cite{CD06}.

A very well-studied class of designs is the class of \textit{balanced incomplete block designs}: they have been investigated for a long time (Fisher's now basic result \cite{Fi40} is from the 1940s) but they are still a rich source of unanswered questions (conjectures about their very existence are being solved only in the last years, see for example Keevash \cite{Ke14} \cite{Ke18}).

In this paper we concentrate on substructures of balanced incomplete block designs, and in particular on a subclass of such designs in which these substructures are especially well-behaved (see Definition~\ref{dewdmb}). Our interest in the problem lies in the connection between these designs and the graph isomorphism problem: block designs are featured in a recent paper by Babai \cite{Ba15} on the proof of the existence of an algorithm that solves the problem in quasipolynomial time (see also Helfgott \cite{HBD17}), and the study of these structures is a likely path to even further improvements in the matter.

In the first section we give definitions and simple properties of designs and subdesigns that are needed in later discussions. In the second section, BIBDs with well-distributed minimal sub-BIBDs are defined and their structural properties (size, intersection numbers of pairs and triples of subdesigns) are investigated using already existing results. In the third section, we describe in more detail the relation between these objects and the graph isomorphism problem in light of Babai's proof.

\subsection{Balanced incomplete block designs}

\begin{definition}\label{debibd}
A balanced incomplete block design (BIBD) is a pair $(V,\mathcal{B})$, where $V$ is a finite set of elements (called vertices) with $|V|=v$ and $\mathcal{B}$ is a collection of subsets of $V$ (called blocks) of size $k>1$ with $|\mathcal{B}|=b$, such that:
\begin{enumerate}[(a)]
\item\label{debibd1} every vertex of $V$ belongs to the same number of blocks $r$;
\item\label{debibd2} every pair of vertices of $V$ is contained in the same number of blocks $\lambda$.
\end{enumerate}
\end{definition}

Condition (\ref{debibd1}) is actually a consequence of the rest of the definition, as shown in Lemma~\ref{lebibdeasy}\ref{lebibdeasy1}, so it is sometimes omitted.

When it is desired to make the parameters explicit, the expressions ``$(v,k,\lambda)$-BIBD'' (as in \cite[Def. 1.2]{St04}) , ``BIBD$(v,b,r,k,\lambda)$'' (as in \cite[Prop. II.1.2]{CD06}), ``$(v,k,\lambda)$-design'' (as in \cite{Li06} \cite[Rem. II.1.5]{CD06}) are used in the literature; we will use the first of the list, to emphasize what kind of design we are talking about and to employ only the necessary parameters in the definition, since $r$ and $b$ depend on the other ones (Lemma~\ref{lebibdeasy}\ref{lebibdeasy2}-\ref{lebibdeasy}\ref{lebibdeasy3}).

There exist some trivial examples of BIBDs: $(V,\emptyset)$ is of course a $(v,k,0)$-BIBD for any choice of $v,k$, and $(V,\{V\})$ is a $(v,v,1)$-BIBD for any $v$; unless explicitly stated, in this paper we ignore these trivial examples. We also mention that in the literature, especially in papers with a more statistical approach, sometimes $\mathcal{B}$ is allowed to count blocks with multiplicity, i.e. there could be distinct blocks $B_{1},B_{2}$ that contain the exact same vertices; again, we suppose that all our blocks are simple unless we expressly state otherwise. Finally, the case $\mathcal{B}=\mathcal{P}_{k}(V)$ (i.e. all subsets of size $k$ are blocks) easily gives a BIBD with $\lambda=\binom{v-2}{k-2}$: we consider this to be trivial, and as such we ignore it; notice that for $k=2$ this is also the only possible nonempty BIBD, so our choice completely rules out the case $k=2$ from now on.

\begin{lemma}\label{lebibdeasy}
Let $(V,\mathcal{B})$ be a $(v,k,\lambda)$-BIBD. Then:
\begin{enumerate}[(a)]
\item\label{lebibdeasy1} condition (\ref{debibd1}) in Definition~\ref{debibd} is redundant, i.e. the fact that every pair of vertices belongs to $\lambda$ blocks of the same size $k$ implies that every vertex belongs to the same number $r$ of blocks;
\item\label{lebibdeasy2} $r=\lambda\frac{v-1}{k-1}$;
\item\label{lebibdeasy3} $b=\lambda\frac{v(v-1)}{k(k-1)}$.
\end{enumerate}
\end{lemma}

\begin{proof}
(\ref{lebibdeasy1}) Fix a vertex $x\in V$. We count in two different ways the number of pairs $(\{x,y\},B)$ such that $y\in V\setminus\{x\}$ and $B\in\mathcal{B}$ is a block containing $x,y$: on one hand, for every choice of $y$ there are $\lambda$ possible $B$; on the other hand, for every choice of $B$ containing $x$ there are $k-1$ possible $y$. If $r_{x}$ is the number of $B$ containing $x$, then we have $\lambda(v-1)=r_{x}(k-1)$: since $v,k,\lambda$ are given parameters independent from $x$, so is $r_{x}$.

(\ref{lebibdeasy2}) From the reasoning above we have $\lambda(v-1)=r(k-1)$, hence $r=\lambda\frac{v-1}{k-1}$.

(\ref{lebibdeasy3}) We count in two different ways the number of pairs $(x,B)$ such that $x\in V$ and $B\in\mathcal{B}$ is a block containing $x$: on one hand, for every choice of $x$ we have $r$ possible $B$; on the other hand, for every choice of $B$ there are $k$ possible $x$. So $rv=bk$, which means that $b=\frac{rv}{k}=\lambda\frac{v(v-1)}{k(k-1)}$.
\end{proof}

There are other types of combinatorial designs that generalize the definition of BIBD and that are of interest to us in the next sections; since we will make multiple references to a couple of them, we now introduce some more definitions.

\begin{definition}\label{depbd}
A pairwise balanced design (PBD) is a pair $(V,\mathcal{B})$, where $V$ is a finite set of elements (called vertices) with $|V|=v$ and $\mathcal{B}$ is a collection of subsets of $V$ (called blocks) with $|\mathcal{B}|=b$, such that:
\begin{enumerate}[(a)]
\item\label{depbd1} every block of $V$ has size $k$ for some $k$ inside a set $K\subseteq\{n\in\mathbb{N}|n\geq 2\}$;
\item\label{depbd2} every pair of vertices of $V$ is contained in the same number of blocks $\lambda$.
\end{enumerate}
\end{definition}

It is important to notice that in some papers the definition actually requires $\lambda$ to be equal to $1$ (see for example \cite{BD16}).

When it is desired to make the parameters explicit, the expressions ``$(v,K,\lambda)$-PBD'' (as in \cite[Def. 7.1]{St04}), ``$\lambda$-PBD'' (as in \cite{Wi75}), ``PBD$(v,K,\lambda)$'', ``$(K,\lambda)$-PBD'' (both quoted in \cite[Def. IV.1.1]{CD06}) are used in the literature; again, we will use the first of the list. When $\lambda=1$, the notations ``$(v,K)$-PBD'' (as in \cite[Def. 7.1]{St04}), ``PBD$(v,K)$'', ``$K$-PBD'' (both quoted in \cite[Def. IV.1.1]{CD06}) are also employed. A BIBD is a particular case of PBD, i.e. the case $|K|=1$ (a $(v,k,\lambda)$-BIBD is the same as a $(v,\{k\},\lambda)$-PBD).

\begin{definition}
A $t$-design is a pair $(V,\mathcal{B})$, where $V$ is a finite set of elements (called vertices) with $|V|=v$ and $\mathcal{B}$ is a collection of subsets of $V$ (called blocks) of size $k$ with $|\mathcal{B}|=b$, such that every subset of $V$ of size $t$ is contained in the same number of blocks $\lambda$.
\end{definition}

When it is desired to make the parameters explicit, the expressions ``$t$-$(v,k,\lambda)$-design'' (as in \cite[Def. 9.1]{St04} \cite{HBD17} \cite[Def. II.4.1]{CD06}) , ``$S_{\lambda}(t,v,k)$'' (as in \cite{RCW75} \cite[Def. II.4.2]{CD06}) are used in the literature; we will use the former. A $(v,k,\lambda)$-BIBD is the same as a $2$-$(v,k,\lambda)$-design; every $t$-$(v,k,\lambda)$-design is also a $t'$-$(v,k,\lambda')$-design for all $t'\leq t$ and some new $\lambda'$ (see \cite[Ex. B.14b]{HBD17}), so every $t$-design with $t\geq 2$ is in particular a BIBD. Lemma~\ref{lebibdeasy}\ref{lebibdeasy1} can be seen also as a consequence of this fact and of the inequality $2\geq 1$.

We also mention a generalization of PBDs and $t$-designs, called \textit{$t$-wise balanced designs} (tBD): as the name suggests, the definition is similar to Definition~\ref{depbd} for PBDs, but condition (\ref{depbd2}) holds for subsets of $V$ of size $t$ instead of pairs (see \cite[Def. 9.33]{St04} \cite[Def. VI.63.1]{CD06}). Clearly, PBDs are tBDs with $t=2$ and $t$-designs are tBDs with $|K|=1$; when in need of making explicit reference to the parameters, we will write ``$t$-$(v,k,\lambda)$-tBD''.

For PBDs, $t$-designs and tBDs we again ignore trivial or extreme cases and we do not allow for repeated blocks unless stated otherwise.

We make the following simple observation.

\begin{lemma}\label{leintblock}
Let $V$ be a $t$-$(v,k,\lambda)$-design with $t\geq 2$ and let $B_{0}\in\mathcal{B}$. Consider the collection $\mathcal{B}'=\{B_{0}\cap B|B\in\mathcal{B}\setminus\{B_{0}\},|B_{0}\cap B|\geq 2\}$ of subsets of $B_{0}$: then $(B_{0},\mathcal{B}')$ is a $t$-$(k,K,\lambda-1)$-tBD for some $K\subseteq\{2,3,\ldots,k-1\}$, possibly trivial and/or with repeated blocks.
\end{lemma}

\begin{proof}
Consider any $t$-tuple of vertices of $B$: since we are inside a $t$-$(v,k,\lambda)$-design, there are $\lambda$ blocks in $V$ that contain the tuple; one of them is $B_{0}$ and the other ones are intersecting $B_{0}$ in at least $t$ vertices, so that their intersections with $B_{0}$ are indeed elements of $\mathcal{B}'$. Therefore the definition of tBD is satisfied with $\lambda-1$.
\end{proof}

The proviso ``possibly trivial and/or with repeated blocks'' is evidently necessary: for example if $\lambda=1$ then blocks intersect only in at most $1$ vertex, so every design obtained in this way would have $\mathcal{B}'=\emptyset$, and in general it is also clear that we might have two distinct blocks $B_{1},B_{2}$ such that $B_{0}\cap B_{1}=B_{0}\cap B_{2}$, thus compelling us to consider multiplicity in the new blocks so as not to lose the condition on $\lambda-1$. It is also to be noted that a $t$-design yields in this way a tBD and not just a smaller $t$-design: instersections of two blocks can have different sizes in a generic $t$-design, so we lose that property in the process.

We recall a well-known result.

\begin{proposition}\label{prfisher}
Every $t$-design with $v\geq k+\lfloor\frac{t}{2}\rfloor$ has $b\geq\binom{v}{\lfloor\frac{t}{2}\rfloor}$; in particular, every BIBD has $b\geq v$ and $v\geq\frac{k(k-1)}{\lambda}+1$. Call a BIBD symmetric if the equality is realized (any of the two): then all the pairwise intersections of blocks have the same size.
\end{proposition}

\begin{proof}
See \cite[Thm. 1]{RCW75}; in particular, the case of a BIBD was solved by Fisher \cite{Fi40}. By Lemma~\ref{lebibdeasy}\ref{lebibdeasy3} and Fisher's result we have $v\geq\frac{k(k-1)}{\lambda}+1$, and equality holds if and only if $b=v$. The statement about block intersections is a special case of \cite[Thm. 4]{RCW75} (when $t=2$).
\end{proof}

In our study of BIBDs, we will focus mostly on the case $\lambda=1$: a $(v,k,1)$-BIBD is also called a \textit{Steiner system}, an expression used especially in the case $k=3$, in which they are called \textit{Steiner triple systems} (see for example \cite[\S 6.2]{St04} \cite[Def. II.2.1]{CD06}). We clarify that the term ``Steiner system'' is used for different classes of designs, either for $t$-$(v,t+1,1)$-designs (as in Steiner's original formulation \cite{St53}) or for $t$-$(v,k,1)$-designs (as it is more common nowadays, see \cite{Ke14} \cite[Def. II.5.1]{CD06}): for the sake of brevity, when we talk about a Steiner system we mean ``a Steiner system in the second broader meaning with $t=2$''.

In order to be possible for a Steiner system with parameters $v,k$ to exist, we must have $k-1|v-1$ by Lemma~\ref{lebibdeasy}\ref{lebibdeasy2} and $k(k-1)|v(v-1)$ by Lemma~\ref{lebibdeasy}\ref{lebibdeasy3}; more generally, necessary  divisibility conditions (also called \textit{admissibility} conditions) for the existence of a PBD with given $v,K$ are that $\lambda(v-1)$ be divisible by $\gcd\{k-1|k\in K\}$ and $\lambda v(v-1)$ be divisible by $\gcd\{k(k-1)|k\in K\}$, as proved by Wilson \cite[Prop. 2.2]{Wi72}. The question of whether these admissibility conditions are sufficient for the existence of a Steiner system or a PBD has positive answer for $v$ large enough, as proved in \cite[Thm. 1]{Wi75} (recent generalizations of this result include \cite{Ke14} \cite{GKLO17} \cite{Ke18}).

\subsection{Sub-BIBDs}

Inside a BIBD, and in general inside a design, it is possible to find a substructure (i.e. a subset of the vertices and a subset of the set of blocks) that shares the same properties of the original larger design.

\begin{definition}
Let $(V,\mathcal{B})$ be a $(v,k,\lambda)$-BIBD. A sub-BIBD is a pair $(V',\mathcal{B}')$ such that $V'\subseteq V$, $\mathcal{B}'\subseteq\mathcal{B}\cap\mathcal{P}(V')$ and $(V',\mathcal{B}')$ is a BIBD; its parameters satisfy $v'\leq v$, $k'=k$, $\lambda'\leq\lambda$.
\end{definition}

The term ``sub-BIBD'', occurring for example in \cite{HMR82} \cite{RS89}, is often replaced in the literature by the vaguer term ``subdesign'', as in \cite[Def. I.1.3]{CD06}, although ``subdesign'' can be applied to substructures of other kinds of designs, as in \cite[\S 9.3.1]{St04}. The term ``flat'' is also employed, especially in the context of subdesigns of PBDs (as in \cite{BD16} \cite[Def. IV.1.30]{CD06}) and when the sets $V$ and blocks $B$ are finite geometric spaces and subspaces (as in \cite[\S 5.2.3 and \S 9.3]{St04}).

We observe also that the notion of subdesign of a PBD is closely linked to the concept of \textit{hole} as in the definition of \textit{incomplete t-wise balanced design} (ItBD) provided in \cite[Def. 9.36]{St04}: in practice, an ItBD is a tBD deprived of a subdesign, leaving a subset $V'\subseteq V$ that has no blocks entirely contained in it, which is called the hole. Clearly, theorems about the structure of holes are also theorems about subdesigns and vice versa; the following lemma for example is stated in terms of holes in the referenced source.

\begin{lemma}[\cite{St04}, Thm. 9.43]\label{lesubboundv}
Let $(V,\mathcal{B})$ be a $(v,k,1)$-BIBD with a sub-BIBD of size $v'$. Then $v\geq (k-1)v'+1$.
\end{lemma}

\begin{proof}
We provide a simpler proof than the one given in \cite[Thm. 9.43]{St04}. Let $V'$ be the set of vertices of the sub-BIBD, fix any vertex $x\not\in V'$ and consider the collection of blocks $B\in\mathcal{B}$ that contain both $x$ and at least one vertex inside $V'$. First, we prove that there are exactly $v'$ of them: since we are in a Steiner system, for any pair $\{x,y\}$ there is exactly one block containing such a pair; moreover, for every pair of distinct vertices $y,y'\in V'$ the two blocks found in this way are also distinct: if they were not, we would have a block containing $x,y,y'$, and this is not possible because there is a unique block containing both $y$ and $y'$ and it has to sit entirely inside $V'$, since $V'$ is a sub-BIBD.

Now that we have proved that there are $v'$ such blocks, the result is easy: all of them intersect in $x$, so they do not pairwise intersect in any other vertex, or we would have a pair of vertices contained in two different blocks in contradiction with the fact that $\lambda=1$; each of them also intersects $V'$ only in one vertex, as we have already pointed out. Therefore for each such block we have $k-2$ vertices outside $V'\cup\{x\}$ contained in it and not shared by any other of these blocks, which means that:
\begin{equation*}
v\geq v'+1+(k-2)v'=(k-1)v'+1
\end{equation*}
as we wanted.
\end{proof}

Sub-BIBDs are not at all rare: it is possible to give recursive constructions of larger BIBDs starting from smaller ones, which lead naturally to the presence of sub-BIBDs. The Doyen-Wilson theorem for example (see \cite{DW73}) states that for any pair of admissible $v'<v$ and any $(v',3,1)$-BIBD there is a $(v,3,1)$-BIBD containing the previous one as sub-BIBD; here with ``admissible'' $v,v'$ we are referring to the divisibility conditions for Steiner triple systems coming from Lemma~\ref{lebibdeasy}, namely $v,v'\equiv 1,3\pmod 6$, and to the bound on sizes coming from Lemma~\ref{lesubboundv}, namely $v\geq 2v'+1$. For examples involving other types of BIBD, see \cite[\S II.7.2]{CD06}.

To give an idea of the frequency of sub-BIBDs, we report here a little computation. The smallest $v$ for which a $(v,k,1)$-BIBD can have proper sub-BIBDs is $(k(k-1)+1)(k-1)+1=k(k^{2}-2k+2)$ (combining Proposition~\ref{prfisher} and Lemma~\ref{lesubboundv}): for $k=3$ this corresponds to $v=15$. There are $80$ non-isomorphic Steiner triple systems with $v=15$ (\cite{CCW17}; for a classification and properties see \cite{MPR83}): of these $80$, $23$ have proper sub-BIBDs; we have verified it using the explicit list provided in a GAP package \cite{NV16} based on Colbourn-Rosa \cite{CR99}. Further direct computations clash with the fast growth of the number $N(v)$ of non-isomorphic Steiner triple systems: for $v=19$ there are 11084874829 such designs, as proven in \cite{KO04}, and in general $(ve^{-5})^{\frac{v^{2}}{6}}\leq N(v)\leq(ve^{-1/2})^{\frac{v^{2}}{6}}$, as proven in \cite{Wi74}. We also checked the BIBDs contained in the Sage package \texttt{sage.combinat.designs.bibd} (v7.6) for $4\leq k\leq 10$ and $k(k^{2}-2k+2)\leq v\leq 1000$: this package constructs $(v,4,1)$-BIBDs following a technique described in \cite[\S 7.4]{St04} and constructs $(v,5,1)$-BIBDs following \cite{Sm04}; out of the $335$ BIBDs thus found, $281$ present proper sub-BIBDs (unevenly distributed among the various $k$: this bias is most likely consequence of the different construction methods in the package rather than reflective of an actual asymmetry in the frequency of subdesigns).

We make another simple observation about sub-BIBDs.

\begin{lemma}\label{leintersbibd}
Let $(V,\mathcal{B})$ be a $(v,k,1)$-BIBD with two sub-BIBDs $(V_{1},\mathcal{B}_{1})$ and $(V_{2},\mathcal{B}_{2})$. Then their intersection $(V_{1}\cap V_{2},\mathcal{B}_{1}\cap\mathcal{B}_{2})$ is also a sub-BIBD: in particular, if we call $|V_{1}\cap V_{2}|=v'$, $|\mathcal{B}_{1}\cap\mathcal{B}_{2}|=b'$, it is either one of the three trivial sub-BIBDs given by $(v',b')=(0,0)$, $(v',b')=(1,0)$, $(v',b')=(k,1)$, or it is nontrivial with $v'>k$.
\end{lemma}

\begin{proof}
Any block in $\mathcal{B}_{1}\cap\mathcal{B}_{2}$ is entirely contained in $V_{1}\cap V_{2}$, so in particular for every pair of vertices $x,y\in V_{1}\cap V_{2}$ by definition we have a unique block $B\in\mathcal{B}_{1}\cap\mathcal{B}_{2}$ containing $x,y$; so $(V_{1}\cap V_{2},\mathcal{B}_{1}\cap\mathcal{B}_{2})$ is again a sub-BIBD.

If $v'=0,1$ the triviality is clear; if $v'\geq 2$ then the sub-BIBD contains at least one block, so $v'\geq k$: if $v'=k$ we have exactly one block taking up the whole sub-BIBD, and the triviality is clear again, while if $v'>k$ (and $\lambda=1$ clearly) our sub-BIBD is nontrivial.
\end{proof}

Thanks to Lemma~\ref{leintersbibd}, we can talk about the sub-BIBD generated by a tuple of vertices of $V$: it will be in particular the intersection of all sub-BIBDs containing those vertices. Alternatively, we can construct it through the following process: consider every pair of vertices of the tuple we are provided with, for each such pair there will be a block in $V$ that contains it; we add those blocks to the collection of blocks of the generated sub-BIBD and every new vertex contained in the blocks to the collection of vertices of the generated sub-BIBD: if while doing this we have added new vertices we repeat this step, until we reach an iteration in which we do not add any other vertex. It is important to notice that a tuple of vertices may generate a trivial sub-BIBD, namely one consisting of only one block.

\section{Well-distributed minimal sub-BIBDs}

As we have already said, for every sufficiently large admissible $v$ there is a $(v,k,1)$-BIBD; also, for every sufficiently large admissible $v'$ there is a $(v,k,1)$-BIBD containing a $(v',k,1)$-BIBD as sub-BIBD. It is possible to do even better and prove that there is a BIBD where subdesigns are very frequent.

\begin{proposition}[\cite{BD16}, Thm. 1.3]\label{prcoversubbibd}
Let $d\in\mathbb{N}$ and $K\subseteq\{n\in\mathbb{N}|n\geq 2\}$. Then there exists a constant $f(d,K)$ such that, for every sufficiently large admissible $v$, there exists a $(v,K,1)$-PBD such that every $d$ vertices generate a subdesign of size at most $f(d,K)$.
\end{proposition}

Since the constant in the proposition does not depend on $v$, by taking $v$ large enough we obtain a PBD (or a Steiner system, if we choose to work with $K=\{k\}$) with an arbitrarily large number of subdesigns that are small in comparison and that are ubiquitous in the original PBD, as every $d$ vertices are contained in one of them. The constant $f(d,K)$ can be made explicit: see \cite[\S 3]{BD16}, where the authors give an expression for $f(d,K)$ involving another constant (called $b$ there, which clashes with our definitions) describing how large an admissible $v$ must be such that a GDD with given $K$ exists; this $b$ in turn is given in \cite{Li06}, with the case $|K|=1$ (the most interesting to us) already considered in \cite{Ch76}.

We observe that $f(d,K)$ is an upper bound on the size of the subdesigns but that they need not be all of the same size: the process described in \cite{BD16} involves among other things the use of the so-called ``Wilson's fundamental construction'' (see \cite{Wi75b}), whose effect is to inflate parts of another kind of design (a \textit{group divisible design}, GDD) according to the different weights given to the vertices; heterogeneous weights lead to a heterogeneous inflation, so that in the end the resulting subdesigns will have different sizes. Hence Proposition~\ref{prcoversubbibd} gives us BIBDs with a lot of sub-BIBDs, but does not ensure any sort of ``regularity'' of distribution of the subdesigns in any way; the rest of the notes will deal with the following case: BIBDs with a lot of sub-BIBDs that are evenly distributed inside the original design.

\subsection{Well-distributed minimal sub-BIBDs: definition and examples}

\begin{definition}\label{dewdmb}
Let $V$ be a BIBD with $\lambda=1$ and let $V'$ be a sub-BIBD. $V'$ is minimal if and only if it has minimal size among all the sub-BIBDs present in $V$ (with $v'>k$).

$V$ has well-distributed minimal sub-BIBDs if and only if:
\begin{enumerate}[(a)]
\item\label{dewdmb1} every vertex of $V$ is contained inside $l$ minimal sub-BIBDs;
\item\label{dewdmb2} every block of $V$ is contained inside $m$ minimal sub-BIBDs.
\end{enumerate}
If $V$ has well-distributed minimal sub-BIBDs, we call $\mathcal{D}$ the collection of such subdesigns and we set $|\mathcal{D}|=n$.
\end{definition}

The elements of $\mathcal{D}$ can be seen as either sets of vertices or sets of blocks, depending on the context.

There actually exist BIBDs satisfying the conditions in Definition~\ref{dewdmb}.

\begin{example}\label{extheone3}
Let $V=\{0,1,2,\ldots,14\}$ and define:
\small
\begin{eqnarray*}
\mathcal{B} & = & \{\{0, 1, 2\}, \{0, 3, 4\}, \{0, 5, 6\}, \{0, 7, 8\}, \{0, 9, 10\}, \{0, 11, 12\}, \{0, 13, 14\}, \\
 & & \{1, 3, 5\}, \{1, 4, 6\}, \{1, 7, 9\}, \{1, 8, 10\}, \{1, 11, 13\}, \{1, 12, 14\}, \{2, 3, 6\}, \\
 & & \{2, 4, 5\}, \{2, 7, 10\}, \{2, 8, 9\}, \{2, 11, 14\}, \{2, 12, 13\}, \{3, 7, 11\}, \{3, 8, 12\}, \\
 & & \{3, 9, 13\}, \{3, 10, 14\}, \{4, 7, 12\}, \{4, 8, 11\}, \{4, 9, 14\}, \{4, 10, 13\}, \{5, 7, 13\}, \\
 & & \{5, 8, 14\}, \{5, 9, 11\}, \{5, 10, 12\}, \{6, 7, 14\}, \{6, 8, 13\}, \{6, 9, 12\}, \{6, 10, 11\}\}
\end{eqnarray*}
\normalsize
Then $(V,\mathcal{B})$ is a $(15,3,1)$-BIBD with $b=35$, $r=7$ that has well-distributed minimal $(7,3,1)$-sub-BIBDs with $n=15$, $l=7$, $m=3$. Of the $23$ non-isomorphic $(15,3,1)$-BIBDs that have subdesigns, $2$ have all blocks covered by $(7,3,1)$-sub-BIBDs and the one in this example is the only one that satisfies Definition~\ref{dewdmb} (i.e. where the subdesigns cover the blocks evenly); this is not the $(15,3,1)$-BIBD contained in the Sage package mentioned in the previous section.
\end{example}

\begin{example}\label{extheone4}
Let $V=\{0,1,2,\ldots,39\}$ and define:
\small
\begin{eqnarray*}
\mathcal{B} & = & \{\{0, 1, 2, 12\}, \{0, 3, 6, 9\}, \{0, 4, 8, 10\}, \{0, 5, 7, 11\}, \{0, 13, 26, 39\}, \\
 & & \{0, 14, 25, 28\}, \{0, 15, 27, 38\}, \{0, 16, 22, 32\}, \{0, 17, 23, 34\}, \{0, 18, 24, 33\}, \\
 & & \{0, 19, 29, 35\}, \{0, 20, 31, 37\}, \{0, 21, 30, 36\}, \{1, 3, 8, 11\}, \{1, 4, 7, 9\}, \\
 & & \{1, 5, 6, 10\}, \{1, 13, 28, 38\}, \{1, 14, 27, 39\}, \{1, 15, 25, 26\}, \{1, 16, 24, 34\}, \\
 & & \{1, 17, 22, 33\}, \{1, 18, 23, 32\}, \{1, 19, 31, 36\}, \{1, 20, 30, 35\}, \{1, 21, 29, 37\}, \\
 & & \{2, 3, 7, 10\}, \{2, 4, 6, 11\}, \{2, 5, 8, 9\}, \{2, 13, 25, 27\}, \{2, 14, 26, 38\}, \\
 & & \{2, 15, 28, 39\}, \{2, 16, 23, 33\}, \{2, 17, 24, 32\}, \{2, 18, 22, 34\}, \{2, 19, 30, 37\}, \\
 & & \{2, 20, 29, 36\}, \{2, 21, 31, 35\}, \{3, 4, 5, 12\}, \{3, 13, 32, 35\}, \{3, 14, 34, 37\}, \\
 & & \{3, 15, 33, 36\}, \{3, 16, 29, 39\}, \{3, 17, 25, 31\}, \{3, 18, 30, 38\}, \{3, 19, 22, 26\}, \\
 & & \{3, 20, 23, 28\}, \{3, 21, 24, 27\}, \{4, 13, 34, 36\}, \{4, 14, 33, 35\}, \{4, 15, 32, 37\}, \\
 & & \{4, 16, 31, 38\}, \{4, 17, 30, 39\}, \{4, 18, 25, 29\}, \{4, 19, 24, 28\}, \{4, 20, 22, 27\}, \\
 & & \{4, 21, 23, 26\}, \{5, 13, 33, 37\}, \{5, 14, 32, 36\}, \{5, 15, 34, 35\}, \{5, 16, 25, 30\}, \\
 & & \{5, 17, 29, 38\}, \{5, 18, 31, 39\}, \{5, 19, 23, 27\}, \{5, 20, 24, 26\}, \{5, 21, 22, 28\}, \\
 & & \{6, 7, 8, 12\}, \{6, 13, 22, 29\}, \{6, 14, 23, 31\}, \{6, 15, 24, 30\}, \{6, 16, 26, 35\}, \\
 & & \{6, 17, 28, 37\}, \{6, 18, 27, 36\}, \{6, 19, 32, 39\}, \{6, 20, 25, 34\}, \{6, 21, 33, 38\}, \\
 & & \{7, 13, 24, 31\}, \{7, 14, 22, 30\}, \{7, 15, 23, 29\}, \{7, 16, 28, 36\}, \{7, 17, 27, 35\}, \\
 & & \{7, 18, 26, 37\}, \{7, 19, 34, 38\}, \{7, 20, 33, 39\}, \{7, 21, 25, 32\}, \{8, 13, 23, 30\}, \\
 & & \{8, 14, 24, 29\}, \{8, 15, 22, 31\}, \{8, 16, 27, 37\}, \{8, 17, 26, 36\}, \{8, 18, 28, 35\}, \\
 & & \{8, 19, 25, 33\}, \{8, 20, 32, 38\}, \{8, 21, 34, 39\}, \{9, 10, 11, 12\}, \{9, 13, 16, 19\}, \\
 & & \{9, 14, 17, 20\}, \{9, 15, 18, 21\}, \{9, 22, 35, 39\}, \{9, 23, 25, 37\}, \{9, 24, 36, 38\}, \\
 & & \{9, 26, 29, 32\}, \{9, 27, 30, 33\}, \{9, 28, 31, 34\}, \{10, 13, 17, 21\}, \{10, 14, 18, 19\}, \\
 & & \{10, 15, 16, 20\}, \{10, 22, 37, 38\}, \{10, 23, 36, 39\}, \{10, 24, 25, 35\}, \{10, 26, 30, 34\}, \\
 & & \{10, 27, 31, 32\}, \{10, 28, 29, 33\}, \{11, 13, 18, 20\}, \{11, 14, 16, 21\}, \{11, 15, 17, 19\}, \\
 & & \{11, 22, 25, 36\}, \{11, 23, 35, 38\}, \{11, 24, 37, 39\}, \{11, 26, 31, 33\}, \{11, 27, 29, 34\}, \\
 & & \{11, 28, 30, 32\}, \{12, 13, 14, 15\}, \{12, 16, 17, 18\}, \{12, 19, 20, 21\}, \{12, 22, 23, 24\}, \\
 & & \{12, 25, 38, 39\}, \{12, 26, 27, 28\}, \{12, 29, 30, 31\}, \{12, 32, 33, 34\}, \{12, 35, 36, 37\}\}
\end{eqnarray*}
\normalsize
Then $(V,\mathcal{B})$ is a $(40,4,1)$-BIBD with $b=130$, $r=13$ that has well-distributed minimal $(13,4,1)$-sub-BIBDs with $n=40$, $l=13$, $m=4$. This is the $(40,4,1)$-BIBD contained in the Sage package mentioned in the previous section; by direct computation, it is the only BIBD satisfying Definition~\ref{dewdmb} among those contained in the package with $k\geq 3$ and $k(k^{2}-2k+2)\leq v'\leq 1000$.
\end{example}

Thanks to our definition, BIBDs with well-distributed minimal sub-BIBDs exhibit interesting properties.

\begin{proposition}\label{prwdmbbibd}
Let $(V,\mathcal{B})$ be a $(v,k,1)$-BIBD with a collection $\mathcal{D}$ of well-distributed minimal $(v',k,1)$-sub-BIBDs. Then:
\begin{enumerate}[(a)]
\item\label{prwdmbbibd1} $(V,\mathcal{D})$ is a $(v,v',m)$-BIBD, where:
\begin{itemize}
\item for any $2\leq i\leq m$, any intersection of $i$ blocks is of size $0,1,k$,
\item for any $m<i\leq l$, any intersection of $i$ blocks is of size $0,1$,
\item for any $i>l$, any $i$ blocks have empty intersection;
\end{itemize}
\item\label{prwdmbbibd2} $(\mathcal{B},\mathcal{D})$ is a $1$-$(b,b',m)$-design, where:
\begin{itemize}
\item for any $2\leq i\leq m$, any intersection of $i$ blocks is of size $0,1$,
\item for any $i>m$, any $i$ blocks have empty intersection.
\end{itemize}
\end{enumerate}
\end{proposition}

\begin{proof}
(\ref{prwdmbbibd1}) The blocks $D\in\mathcal{D}$, seen here as sets of vertices of the sub-BIBDs of $(V,\mathcal{B})$, have all the same size $v'$; moreover, every pair of vertices $x,y\in V$ is contained in a sub-BIBD $D$ if and only if $D$ contains the whole (unique) block containing $x,y$, so by Definition~\ref{dewdmb}\ref{dewdmb2} every pair $x,y$ is contained in the same number $m$ of $D\in\mathcal{D}$. Hence $(V,\mathcal{D})$ is a $(v,v',m)$-BIBD.

By Lemma~\ref{leintersbibd}, the intersection of any two $D_{1},D_{2}\in\mathcal{D}$ is a sub-BIBD in the original $(V,\mathcal{B})$: then by minimality the only possibilities for this intersection are the trivial sub-BIBDs of size $0,1,k$ mentioned in the lemma. Obviously, the same happens for the intersection of any $i$ elements of $\mathcal{D}$; since every pair of vertices is contained in $m$ sub-BIBDs, if $i>m$ then it is not possible to have intersection size $>1$, while, since every vertex is contained in $l$ sub-BIBDs, if $i>l$ then the intersection of $i$ blocks must be empty.

(\ref{prwdmbbibd2}) The blocks $D\in\mathcal{D}$, seen now instead as sets of elements of $\mathcal{B}$, have all the same size $b'=\frac{v'(v'-1)}{k(k-1)}$ by Lemma~\ref{lebibdeasy}\ref{lebibdeasy3}; every $B\in\mathcal{B}$ is contained in $m$ blocks $D$ by Definition~\ref{dewdmb}\ref{dewdmb2}, hence $(\mathcal{B},\mathcal{D})$ is a $1$-$(b,b',m)$-design. The facts about the intersection sizes are easy consequences of being a $1$-design, or easy consequences of part (\ref{prwdmbbibd1}): a $k$-intersection in part (\ref{prwdmbbibd1}) corresponds to a $1$-intersection here, and a $0,1$-intersection in part (\ref{prwdmbbibd1}) becomes a $0$-intersection here.
\end{proof}

It is interesting to see also what happens to each $D$. By Lemma~\ref{leintblock} and Proposition~\ref{prwdmbbibd}\ref{prwdmbbibd1}, the collection $\mathcal{D}'$ of pairwise intersections $D_{0}\cap D$ for a fixed $D_{0}$ forms a PBD $(D_{0},\mathcal{D}')$; the properties of this BIBD imply that the induced PBD is rather special: $(D_{0},\mathcal{D}')$ is a BIBD with $\lambda=m-1$ with repeated blocks, even better, it is the union of $m-1$ copies of one BIBD with $\lambda=1$, namely the BIBD $(D_{0},\mathcal{P}_{k}(D_{0})\cap\mathcal{B})$ since the pairwise intersections of subdesigns define the blocks of the original $(V,\mathcal{B})$ itself (if $m>1$; if $m=1$, the collection $\mathcal{D}'$ is empty and the induced BIBD is trivial).

Proposition~\ref{prwdmbbibd} implies that one of the three parameters $l,m,n$ in Definition~\ref{dewdmb} determines the other two. In the future we will mostly use only $m$ and let the other ones be function of $m$.

\begin{corollary}\label{conl}
Let $V$ be a $(v,k,1)$-BIBD with a collection $\mathcal{D}$ of well-distributed minimal $(v',k,1)$-sub-BIBDs. Then:
\begin{enumerate}[(a)]
\item\label{conl1} $n=m\frac{v(v-1)}{v'(v'-1)}$;
\item\label{conl2} $l=m\frac{v-1}{v'-1}$.
\end{enumerate}
\end{corollary}

\begin{proof}
By Proposition~\ref{prwdmbbibd}\ref{prwdmbbibd1}, $(V,\mathcal{D})$ is a $(v,v',m)$-BIBD: $n$ is the number of blocks and $l$ is the number of blocks containing one vertex. The results then descend from Lemma~\ref{lebibdeasy}\ref{lebibdeasy3} and Lemma~\ref{lebibdeasy}\ref{lebibdeasy2}, respectively.
\end{proof}

\begin{example}\label{exconstr}
Let $V$ be a $(v,k,1)$-BIBD with well-distributed minimal sub-BIBDs, and suppose that there exists a $(w,v,1)$-BIBD: then we can easily build a $(w,k,1)$-BIBD. In fact, if $(V,\mathcal{B}_{V})$ is a $(v,k,1)$-BIBD and $(W,\mathcal{B}_{W})$ is a $(w,v,1)$-BIBD, we can do the following: for each $B\in\mathcal{B}_{W}$, we choose any bijection $\varphi_{B}:V\rightarrow B$ and replace the block $B$ (seen as an element of $\mathcal{B}_{W}$) with the subsets $\varphi_{B}(B')$ of size $k$ for all $B'\in\mathcal{B}_{V}$. Then the pair $\left(W,\bigcup_{B\in\mathcal{B}_{W}}\{\varphi_{B}(B')|B'\in\mathcal{B}_{V}\}\right)$ is a $(w,k,1)$-BIBD: indeed, every pair of vertices $x,y\in W$ is contained in a unique block $B\in\mathcal{B}_{W}$ and in turn, after having fixed such $B$, the preimages $\varphi^{-1}(x),\varphi^{-1}(y)$ are contained in a unique block $B'\in\mathcal{B}_{V}$ in $V$; therefore condition (\ref{debibd2}) in Definition~\ref{debibd} is satisfied again in the new collection of blocks of $W$, with $\lambda=1$ as before.

Is it true that the new $W$ has well-distributed minimal sub-BIBDs? The sub-BIBDs coming from $V$ are naturally transferred to $W$: every vertex in $W$ is contained in $\frac{w-1}{v-1}$ copies of $V$ (by Lemma~\ref{lebibdeasy}\ref{lebibdeasy2}), and for each copy there are $l=m\frac{v-1}{v'-1}$ sub-BIBDs containing the vertex (by Corollary~\ref{conl}\ref{conl2}); on the other hand, every block of size $k$ is contained in only one copy of $V$ and, inside such copy, in $m$ sub-BIBDs. Therefore, these sub-BIBDs are still well-distributed inside the new $W$. However, it is not always true that they are still minimal: it is possible to obtain a smaller sub-BIBD consisting of blocks coming all from different copies of $V$ (the intersection of these sub-BIBDs and the old ones would at most be trivial $(k,k,1)$-sub-BIBDs, i.e. single blocks); also, for the same reason it is not guaranteed that we do not have other sub-BIBDs of the same size, so that while the ones coming from $V$ are well-distributed it is not sure that the whole collection of sub-BIBDs of size $v'$ in $W$ is well-distributed.

In the case of Examples~\ref{extheone3}-\ref{extheone4}, since $v'=k(k-1)+1$ the problem of minimality would not pose itself: since a $(k(k-1)+1,k,1)$-BIBD is symmetric it has the smallest size among all Steiner systems with that $k$ (Proposition~\ref{prfisher}), therefore in building larger designs satisfying Definition~\ref{dewdmb} using the BIBDs in those examples we would only have to check well-distribution.
\end{example}

\subsection{Intersections of two and three sub-BIBDs}

Once we have our Definition~\ref{dewdmb} and we know that well-distributed sub-BIBDs behave in a way as larger blocks, as shown by Proposition~\ref{prwdmbbibd}, we can start investigating properties of the intersection of sub-BIBDs. This is a well-trodden path of research: the properties of $2$-designs allow to deduce results about intersection of $2$ and $3$ blocks inside them (see for instance \cite{Ma53} \cite{RCW75} \cite{Ca88} to witness how the exploitation works in different fashions).

We will need some simple observations about blocks and well-distributed sub-BIBDs.

\begin{lemma}\label{lewdmbeasy}
Let $V$ be a $(v,k,1)$-BIBD. Then, for any fixed block $B_{1}\in\mathcal{B}$:
\begin{eqnarray}
|\{B_{2}||B_{1}\cap B_{2}|=1\}| & = & \frac{k(v-k)}{k-1} \label{lewdmbeasyb1} \\
|\{B_{2}||B_{1}\cap B_{2}|=0\}| & = & \frac{(v-k^{2}+k-1)(v-k)}{k(k-1)} \label{lewdmbeasyb0}
\end{eqnarray}

Let $V$ be a $(v,k,1)$-BIBD with well-distributed minimal $(v',k,1)$-sub-BIBDs; fix one of them, say $V_{1}$. Then, for any given vertex $x\in V_{1}$ and block $B\subseteq V_{1}$:
\small
\begin{eqnarray}
I_{k}=|\{V_{2}||V_{1}\cap V_{2}|=k\}| & = & \frac{v'(v'-1)(m-1)}{k(k-1)} \label{lewdmbeasyk} \\
|\{V_{2}||V_{1}\cap V_{2}|=k,x\in V_{2}\}| & = & \frac{(v'-1)(m-1)}{k-1} \label{lewdmbeasykx} \\
|\{V_{2}||V_{1}\cap V_{2}|=k,x\not\in V_{2}\}| & = & \frac{(v'-1)(v'-k)(m-1)}{k(k-1)} \label{lewdmbeasyknx} \\
|\{V_{2}||V_{1}\cap V_{2}|=k,|B\cap V_{2}|=1\}| & = & \frac{k(v'-k)(m-1)}{k-1} \label{lewdmbeasykb1} \\
|\{V_{2}||V_{1}\cap V_{2}|=k,|B\cap V_{2}|=0\}| & = & \frac{(v'-k^{2}+k-1)(v'-k)(m-1)}{k(k-1)} \label{lewdmbeasykb0} \\
|\{V_{2}|V_{1}\cap V_{2}=\{x\}\}| & = & m\left(\frac{v-1}{v'-1}-\frac{v'-1}{k-1}\right)+\frac{v'-k}{k-1} \label{lewdmbeasyx} \\
I_{1}=|\{V_{2}||V_{1}\cap V_{2}|=1\}| & = & v'\left(m\left(\frac{v-1}{v'-1}-\frac{v'-1}{k-1}\right)+\frac{v'-k}{k-1}\right) \label{lewdmbeasy1} \\
I_{0}=|\{V_{2}|V_{1}\cap V_{2}=\emptyset\}| & = & m\left(\frac{(v-1)(v-v'^{2})}{v'(v'-1)}+\frac{v'(v'-1)}{k}\right)- \nonumber \\
 & & - \ \frac{(v'-1)(v'-k)}{k} \label{lewdmbeasy0}
\end{eqnarray}
\normalsize
\end{lemma}

\begin{proof}
By Lemma~\ref{lebibdeasy}\ref{lebibdeasy2} there are $\frac{v-1}{k-1}-1$ blocks $B_{2}$ distinct from $B_{1}$ and containing a given vertex of $B_{1}$; also, any two blocks containing two given vertices inside $B_{1}$ must be distinct, because otherwise there would be a block intersecting $B_{1}$ in more than $1$ vertex. Therefore, as there are $k$ vertices in $B_{1}$, we get \eqref{lewdmbeasyb1}. Knowing by Lemma~\ref{lebibdeasy}\ref{lebibdeasy3} that there are $\frac{v(v-1)}{k(k-1)}$ blocks in $V$, and subtracting $B_{1}$ and the blocks considered in \eqref{lewdmbeasyb1}, we get \eqref{lewdmbeasyb0}.

By Definition~\ref{dewdmb}\ref{dewdmb2}, each of the $b'$ blocks $B$ inside $V_{1}$ is contained in $m-1$ sub-BIBDs distinct from $V_{1}$; also, any two sub-BIBDs containing two given blocks inside $V_{1}$ must be distinct, because otherwise there would be a sub-BIBD intersecting $V_{1}$ in more than $k$ vertices. Therefore, using Lemma~\ref{lebibdeasy}\ref{lebibdeasy3}, we obtain \eqref{lewdmbeasyk}.

Lemma~\ref{lebibdeasy}\ref{lebibdeasy2} gives the number of blocks inside $V_{1}$ that intersect a fixed $x$, so reasoning as we did above we obtain \eqref{lewdmbeasykx}. Taking the difference between \eqref{lewdmbeasyk} and \eqref{lewdmbeasykx}, we obtain \eqref{lewdmbeasyknx}.

Using \eqref{lewdmbeasyb1} and \eqref{lewdmbeasyb0} on $V_{1}$ and remembering that there are $m-1$ sub-BIBDs for each block, we obtain \eqref{lewdmbeasykb1} and \eqref{lewdmbeasykb0}.

By Definition~\ref{dewdmb}\ref{dewdmb1}, there are $l$ sub-BIBDs containing any given $x$: subtracting $V_{1}$ itself and the sub-BIBDs counted in \eqref{lewdmbeasykx}, and using Corollary~\ref{conl}\ref{conl2}, we obtain \eqref{lewdmbeasyx}; \eqref{lewdmbeasy1} is obtained simply multiplying \eqref{lewdmbeasyx} by the number of vertices $v'$ in a single subdesign. Finally we consider all the $n$ sub-BIBDs and we recall Corollary~\ref{conl}\ref{conl1}; subtracting $V_{1}$, the $V_{2}$ considered in \eqref{lewdmbeasyk}, and the $V_{2}$ considered in \eqref{lewdmbeasy1}, we get \eqref{lewdmbeasy0}.
\end{proof}

Then we deduce some results about intersections of three sub-BIBDs: these results are actually valid for intersections of blocks inside any BIBD, and they apply in our case thanks to Proposition~\ref{prwdmbbibd}\ref{prwdmbbibd1}. We follow Majumdar \cite{Ma53}, who credits Connor \cite{Co52}.

\begin{lemma}\label{lemaj}
Let $V$ be a $(v,k,1)$-BIBD with a collection $\mathcal{D}$ of well-distributed minimal $(v',k,1)$-sub-BIBDs. Fix two sub-BIBDs $D_{1},D_{2}\in\mathcal{D}$ and let $i\in\{0,1,k\}$ be the size of their intersection; for any $D\neq D_{1},D_{2}$ define $i_{1,D}=|D\cap(D_{1}\setminus D_{2})|$, $i_{2,D}=|D\cap D_{1}\cap D_{2}|$, $i_{3,D}=|D\cap(D_{2}\setminus D_{1})|$. Then:
\small
\begin{eqnarray}
\sum_{D}i_{1,D} & = & (v'-i)\left(m\frac{v-1}{v'-1}-1\right) \label{lemaj1} \ = \ \sum_{D}i_{3,D} \\
\sum_{D}i_{1,D}^{2} & = & (v'-i)\left(m\left(\frac{v-1}{v'-1}+v'-i-1\right)-v'+i\right) \label{lemaj11} \ = \ \sum_{D}i_{3,D}^{2} \\
\sum_{D}i_{2,D} & = & i\left(m\frac{v-1}{v'-1}-2\right) \label{lemaj2} \\
\sum_{D}i_{1,D}i_{2,D} & = & i(v'-i)(m-1) \label{lemaj12} \ = \ \sum_{D}i_{2,D}i_{3,D} \\
\sum_{D}i_{1,D}i_{3,D} & = & (v'-i)^{2}m \label{lemaj13}
\end{eqnarray}
\normalsize
\end{lemma}

\begin{proof}
Summing all the $i_{1,D}$ is the same as counting pairs $(x,D)$ such that $D\neq D_{1},D_{2}$ and $x\in D\cap(D_{1}\setminus D_{2})$: there are $v'-i$ such vertices and $l-1$ such subdesigns for each vertex, so by Corollary~\ref{conl}\ref{conl2} we obtain \eqref{lemaj1}. An analogous reasoning holds for \eqref{lemaj2}, although in this case there are only $i$ vertices and we have to exclude both $D_{1}$ and $D_{2}$.

Summing all the $i_{1,D}^{2}$ is the same as counting triples $(x,y,D)$ such that $D\neq D_{1},D_{2}$ and $x,y\in D\cap(D_{1}\setminus D_{2})$. To account for the case $x=y$ we can just use \eqref{lemaj1}; when $x\neq y$ we are simply considering every (ordered) pair of vertices $x,y$ and every sub-BIBD distinct from $D_{1}$ that contains it: there are $(v'-i)(v'-i-1)$ possible ordered pairs of distinct vertices and $m-1$ sub-BIBDs for each choice, and summing them with \eqref{lemaj1} we obtain \eqref{lemaj11}. Analogously, we get \eqref{lemaj12} (resp. \eqref{lemaj13}) by considering $i(v'-i)$ pairs of vertices (resp. $(v'-i)^{2}$ pairs) and $m-1$ sub-BIBDs for each choice (resp. $m$).
\end{proof}

Using Lemma~\ref{lemaj}, we can obtain interesting structural results about the frequency of certain combinations of intersections of three sub-BIBDs. These are essentially based on the fact that minimal sub-BIBDs have only three possible intersection sizes.

Let us start with the case $i=1$. For any two $D_{1},D_{2}$ with intersection size $1$, we define below certain numbers $a^{(i)}_{D_{1}D_{2}}$; for the sake of brevity, when we say ``$a^{(i)}_{D_{1}D_{2}}=(x_{1},x_{2},x_{3})$'' we mean that $a^{(i)}_{D_{1}D_{2}}$ is equal to the number of $D$ such that $|D\cap(D_{1}\setminus D_{2})|=x_{1}$, $|D\cap D_{1}\cap D_{2}|=x_{2}$, $|D\cap(D_{2}\setminus D_{1})|=x_{3}$.
\begin{align*}
a^{(1)}_{D_{1}D_{2}}=& \ (k,0,k) & a^{(7)}_{D_{1}D_{2}}=& \ (k-1,1,0) \\
a^{(2)}_{D_{1}D_{2}}=& \ (k-1,1,k-1) & a^{(8)}_{D_{1}D_{2}}=& \ (0,1,k-1) \\
a^{(3)}_{D_{1}D_{2}}=& \ (k,0,1) & a^{(9)}_{D_{1}D_{2}}=& \ (1,0,1) \\
a^{(4)}_{D_{1}D_{2}}=& \ (1,0,k) & a^{(10)}_{D_{1}D_{2}}=& \ (1,0,0) \\
a^{(5)}_{D_{1}D_{2}}=& \ (k,0,0) & a^{(11)}_{D_{1}D_{2}}=& \ (0,0,1) \\
a^{(6)}_{D_{1}D_{2}}=& \ (0,0,k) & a^{(12)}_{D_{1}D_{2}}=& \ (0,1,0) \\
 & & a^{(13)}_{D_{1}D_{2}}=& \ (0,0,0)
\end{align*}
It is clear from the fact that any pairwise intersection has size $\in\{0,1,k\}$ that the list exhausts all possibilities, i.e. every $D\neq D_{1},D_{2}$ is counted in one of the $13$ numbers above.

\begin{proposition}\label{prsyst1}
Let $V$ be a $(v,k,1)$-BIBD that has well-distributed minimal $(v',k,1)$-sub-BIBDs. Fix two such subdesigns $D_{1},D_{2}$ with $|D_{1}\cap D_{2}|=1$ (provided that they exist); then:
\small
\begin{eqnarray*}
a^{(5)}_{D_{1}D_{2}} & = & \frac{(m-1)(v'-1)(v'-k)}{k(k-1)}-a^{(1)}_{D_{1}D_{2}}-a^{(3)}_{D_{1}D_{2}} \\
a^{(6)}_{D_{1}D_{2}} & = & \frac{(m-1)(v'-1)(v'-k)}{k(k-1)}-a^{(1)}_{D_{1}D_{2}}-a^{(4)}_{D_{1}D_{2}} \\
a^{(7)}_{D_{1}D_{2}} & = & \frac{(m-1)(v'-1)}{k-1}-a^{(2)}_{D_{1}D_{2}} \ = \ a^{(8)}_{D_{1}D_{2}} \\
a^{(9)}_{D_{1}D_{2}} & = & m(v'-1)^{2}-k^{2}a^{(1)}_{D_{1}D_{2}}-(k-1)^{2}a^{(2)}_{D_{1}D_{2}}-ka^{(3)}_{D_{1}D_{2}}-ka^{(4)}_{D_{1}D_{2}} \\
a^{(10)}_{D_{1}D_{2}} & = & m\left(v-1-\frac{k(v'-1)^{2}}{k-1}\right)+\frac{(v'-1)(v'-k)}{k-1}+ \\
 & & + \ k^{2}a^{(1)}_{D_{1}D_{2}}+(k-1)^{2}a^{(2)}_{D_{1}D_{2}}+ka^{(3)}_{D_{1}D_{2}}+(k-1)a^{(4)}_{D_{1}D_{2}} \\
a^{(11)}_{D_{1}D_{2}} & = & m\left(v-1-\frac{k(v'-1)^{2}}{k-1}\right)+\frac{(v'-1)(v'-k)}{k-1}+ \\
 & & + \ k^{2}a^{(1)}_{D_{1}D_{2}}+(k-1)^{2}a^{(2)}_{D_{1}D_{2}}+(k-1)a^{(3)}_{D_{1}D_{2}}+ka^{(4)}_{D_{1}D_{2}} \\
a^{(12)}_{D_{1}D_{2}} & = & m\left(\frac{v-1}{v'-1}-2\frac{v'-1}{k-1}\right)+2\frac{v'-k}{k-1}+a^{(2)}_{D_{1}D_{2}} \\
a^{(13)}_{D_{1}D_{2}} & = & m\left(\frac{v(v-1)}{v'(v'-1)}+(v'-1)^{2}+\frac{v-1}{v'-1}-\right. \\
 & & - \left.2v'\left(\frac{v-1}{v'-1}-\frac{v'-1}{k}\right)\right)-\frac{2(v'-1)(v'-k)}{k}-\\
 & & - \ (k^{2}-1)a^{(1)}_{D_{1}D_{2}}-(k-1)^{2}a^{(2)}_{D_{1}D_{2}}-(k-1)a^{(3)}_{D_{1}D_{2}}-(k-1)a^{(4)}_{D_{1}D_{2}}
\end{eqnarray*}
\normalsize
\end{proposition}

\begin{proof}
We compute the sums on the LHS of the equations in Lemma~\ref{lemaj} in this situation. By our definitions:
\small
\begin{eqnarray*}
\sum_{D}i_{1,D} & = & ka^{(1)}_{D_{1}D_{2}}+ (k-1)a^{(2)}_{D_{1}D_{2}}+ ka^{(3)}_{D_{1}D_{2}}+ a^{(4)}_{D_{1}D_{2}}+ ka^{(5)}_{D_{1}D_{2}}+ \\
 & & + \ (k-1)a^{(7)}_{D_{1}D_{2}}+ a^{(9)}_{D_{1}D_{2}}+ a^{(10)}_{D_{1}D_{2}} \\
\sum_{D}i_{3,D} & = & ka^{(1)}_{D_{1}D_{2}}+ (k-1)a^{(2)}_{D_{1}D_{2}}+ a^{(3)}_{D_{1}D_{2}}+ ka^{(4)}_{D_{1}D_{2}}+ ka^{(6)}_{D_{1}D_{2}}+ \\
 & & + \ (k-1)a^{(8)}_{D_{1}D_{2}}+ a^{(9)}_{D_{1}D_{2}}+ a^{(11)}_{D_{1}D_{2}} \\
\sum_{D}i_{1,D}^{2} & = & k^{2}a^{(1)}_{D_{1}D_{2}}+ (k-1)^{2}a^{(2)}_{D_{1}D_{2}}+ k^{2}a^{(3)}_{D_{1}D_{2}}+ a^{(4)}_{D_{1}D_{2}}+ k^{2}a^{(5)}_{D_{1}D_{2}}+ \\
 & & + \ (k-1)^{2}a^{(7)}_{D_{1}D_{2}}+ a^{(9)}_{D_{1}D_{2}}+ a^{(10)}_{D_{1}D_{2}} \\
\sum_{D}i_{3,D}^{2} & = & k^{2}a^{(1)}_{D_{1}D_{2}}+ (k-1)^{2}a^{(2)}_{D_{1}D_{2}}+ a^{(3)}_{D_{1}D_{2}}+ k^{2}a^{(4)}_{D_{1}D_{2}}+ k^{2}a^{(6)}_{D_{1}D_{2}}+ \\
 & & + \ (k-1)^{2}a^{(8)}_{D_{1}D_{2}}+ a^{(9)}_{D_{1}D_{2}}+ a^{(11)}_{D_{1}D_{2}} \\
\sum_{D}i_{2,D} & = & a^{(2)}_{D_{1}D_{2}}+ a^{(7)}_{D_{1}D_{2}}+ a^{(8)}_{D_{1}D_{2}}+ a^{(12)}_{D_{1}D_{2}} \\
\sum_{D}i_{1,D}i_{2,D} & = & (k-1)a^{(2)}_{D_{1}D_{2}}+ (k-1)a^{(7)}_{D_{1}D_{2}} \\
\sum_{D}i_{2,D}i_{3,D} & = & (k-1)a^{(2)}_{D_{1}D_{2}}+ (k-1)a^{(8)}_{D_{1}D_{2}} \\
\sum_{D}i_{1,D}i_{3,D} & = & k^{2}a^{(1)}_{D_{1}D_{2}}+ (k-1)^{2}a^{(2)}_{D_{1}D_{2}}+ ka^{(3)}_{D_{1}D_{2}}+ ka^{(4)}_{D_{1}D_{2}}+ a^{(9)}_{D_{1}D_{2}}
\end{eqnarray*}
\normalsize
Since the $a^{(i)}_{D_{1}D_{2}}$ count all $D\neq D_{1},D_{2}$, we also have:
\begin{equation*}
m\frac{v(v-1)}{v'(v'-1)}-2 \ = \ \sum_{i=1}^{13}a^{(i)}_{D_{1}D_{2}}
\end{equation*}
Then using $i=1$ in Lemma~\ref{lemaj} and inverting the linear system we obtain the result.
\end{proof}

Then we handle the case $i=k$. For any two $D_{1},D_{2}$ with intersection size $k$, we define the following numbers:
\begin{align*}
c^{(1)}_{D_{1}D_{2}}=& \ (k,0,k) & c^{(7)}_{D_{1}D_{2}}=& \ (k-1,1,0) \\
c^{(2)}_{D_{1}D_{2}}=& \ (k-1,1,k-1) & c^{(8)}_{D_{1}D_{2}}=& \ (0,1,k-1) \\
c^{(3)}_{D_{1}D_{2}}=& \ (k,0,1) & c^{(9)}_{D_{1}D_{2}}=& \ (1,0,1) \\
c^{(4)}_{D_{1}D_{2}}=& \ (1,0,k) & c^{(10)}_{D_{1}D_{2}}=& \ (1,0,0) \\
c^{(5)}_{D_{1}D_{2}}=& \ (k,0,0) & c^{(11)}_{D_{1}D_{2}}=& \ (0,0,1) \\
c^{(6)}_{D_{1}D_{2}}=& \ (0,0,k) & c^{(12)}_{D_{1}D_{2}}=& \ (0,1,0) \\
 & & c^{(13)}_{D_{1}D_{2}}=& \ (0,0,0)
\end{align*}
This time, the list exhausts all $D$ that do not intersect $D_{1},D_{2}$ in exactly $D_{1}\cap D_{2}$; of course, we already know by Definition~\ref{dewdmb}\ref{dewdmb2} that there are $m-2$ sub-BIBDs $D$ with $D\cap D_{1}=D\cap D_{2}=D_{1}\cap D_{2}$.

\begin{proposition}\label{prsystk}
Let $V$ be a $(v,k,1)$-BIBD that has well-distributed minimal $(v',k,1)$-sub-BIBDs. Fix two such subdesigns $D_{1},D_{2}$ with $|D_{1}\cap D_{2}|=k$ (provided that they exist); then:
\small
\begin{eqnarray*}
c^{(5)}_{D_{1}D_{2}} & = & \frac{(m-1)(v'-k)(v'-k^{2}+k-1)}{k(k-1)}-c^{(1)}_{D_{1}D_{2}}-c^{(3)}_{D_{1}D_{2}} \\
c^{(6)}_{D_{1}D_{2}} & = & \frac{(m-1)(v'-k)(v'-k^{2}+k-1)}{k(k-1)}-c^{(1)}_{D_{1}D_{2}}-c^{(4)}_{D_{1}D_{2}} \\
c^{(7)}_{D_{1}D_{2}} & = & \frac{k(m-1)(v'-k)}{k-1}-c^{(2)}_{D_{1}D_{2}} \ = \ c^{(8)}_{D_{1}D_{2}} \\
c^{(9)}_{D_{1}D_{2}} & = & m(v'-k)^{2}-k^{2}c^{(1)}_{D_{1}D_{2}}-(k-1)^{2}c^{(2)}_{D_{1}D_{2}}-kc^{(3)}_{D_{1}D_{2}}-kc^{(4)}_{D_{1}D_{2}} \\
c^{(10)}_{D_{1}D_{2}} & = & m(v'-k)\left(\frac{v-1}{v'-1}-\frac{kv'-k^{2}+k-1}{k-1}\right)+\frac{(v'-k)^{2}}{k-1}+ \\
 & & + \ k^{2}c^{(1)}_{D_{1}D_{2}}+(k-1)^{2}c^{(2)}_{D_{1}D_{2}}+kc^{(3)}_{D_{1}D_{2}}+(k-1)c^{(4)}_{D_{1}D_{2}} \\
c^{(11)}_{D_{1}D_{2}} & = & m(v'-k)\left(\frac{v-1}{v'-1}-\frac{kv'-k^{2}+k-1}{k-1}\right)+\frac{(v'-k)^{2}}{k-1}+ \\
 & & + \ k^{2}c^{(1)}_{D_{1}D_{2}}+(k-1)^{2}c^{(2)}_{D_{1}D_{2}}+(k-1)c^{(3)}_{D_{1}D_{2}}+kc^{(4)}_{D_{1}D_{2}} \\
c^{(12)}_{D_{1}D_{2}} & = & mk\left(\frac{v-1}{v'-1}-2\frac{v'-k}{k-1}-1\right)+2k\frac{v'-k}{k-1}+c^{(2)}_{D_{1}D_{2}} \\
c^{(13)}_{D_{1}D_{2}} & = & m\left(1-k+\frac{v(v-1)}{v'(v'-1)}+(v'-k)^{2}+k\frac{v-1}{v'-1}-\right. \\
 & & \left.-2v'\left(\frac{v-1}{v'-1}-\frac{v'-1}{k}\right)\right)-\frac{2(v'-1)(v'-k)}{k}- \\
 & & - \ (k^{2}-1)c^{(1)}_{D_{1}D_{2}}-(k-1)^{2}c^{(2)}_{D_{1}D_{2}}-(k-1)c^{(3)}_{D_{1}D_{2}}-(k-1)c^{(4)}_{D_{1}D_{2}}
\end{eqnarray*}
\normalsize
\end{proposition}

\begin{proof}
We compute again the sums on the LHS of Lemma~\ref{lemaj}; the expressions look exactly as in Proposition~\ref{prsyst1} with $c^{(i)}_{D_{1}D_{2}}$ instead of $a^{(i)}_{D_{1}D_{2}}$, except for:
\begin{equation*}
\sum_{D}i_{2,D}=c^{(2)}_{D_{1}D_{2}}+ c^{(7)}_{D_{1}D_{2}}+ c^{(8)}_{D_{1}D_{2}}+ c^{(12)}_{D_{1}D_{2}}+ k(m-2)
\end{equation*}
since in this case we have to sum the contribution from the subdesigns that intersect $D_{1}$ and $D_{2}$ exactly in $D_{1}\cap D_{2}$. As we already mentioned we also have:
\begin{equation*}
m\left(\frac{v(v-1)}{v'(v'-1)}-1\right) \ = \ \sum_{i=1}^{13}c^{(i)}_{D_{1}D_{2}}
\end{equation*}
Then we use Lemma~\ref{lemaj} with $i=k$ and we invert the system.
\end{proof}

Finally, we deal with the case $i=0$. For any two disjoint $D_{1},D_{2}$ we define:
\begin{align*}
e^{(1)}_{D_{1}D_{2}}=& \ (k,0,k) & e^{(6)}_{D_{1}D_{2}}=& \ (1,0,1) \\
e^{(2)}_{D_{1}D_{2}}=& \ (k,0,1) & e^{(7)}_{D_{1}D_{2}}=& \ (1,0,0) \\
e^{(3)}_{D_{1}D_{2}}=& \ (1,0,k) & e^{(8)}_{D_{1}D_{2}}=& \ (0,0,1) \\
e^{(4)}_{D_{1}D_{2}}=& \ (k,0,0) & e^{(9)}_{D_{1}D_{2}}=& \ (0,0,0) \\
e^{(5)}_{D_{1}D_{2}}=& \ (0,0,k) & &
\end{align*}
Again, all $D\neq D_{1},D_{2}$ are counted in one of the numbers above.

\begin{proposition}\label{prsyst0}
Let $V$ be a $(v,k,1)$-BIBD that has well-distributed minimal $(v',k,1)$-sub-BIBDs. Fix two such subdesigns $D_{1},D_{2}$ that are disjoint from each other (provided that they exist); then:
\small
\begin{eqnarray*}
e^{(4)}_{D_{1}D_{2}} & = & \frac{(m-1)v'(v'-1)}{k(k-1)}-e^{(1)}_{D_{1}D_{2}}-e^{(2)}_{D_{1}D_{2}} \\
e^{(5)}_{D_{1}D_{2}} & = & \frac{(m-1)v'(v'-1)}{k(k-1)}-e^{(1)}_{D_{1}D_{2}}-e^{(3)}_{D_{1}D_{2}} \\
e^{(6)}_{D_{1}D_{2}} & = & mv'^{2}-k^{2}e^{(1)}_{D_{1}D_{2}}-ke^{(2)}_{D_{1}D_{2}}-ke^{(3)}_{D_{1}D_{2}} \\
e^{(7)}_{D_{1}D_{2}} & = & mv'\left(\frac{v-1}{v'-1}-\frac{kv'-1}{k-1}\right)+\frac{v'(v'-k)}{k-1} +k^{2}e^{(1)}_{D_{1}D_{2}}+ke^{(2)}_{D_{1}D_{2}}+(k-1)e^{(3)}_{D_{1}D_{2}} \\
e^{(8)}_{D_{1}D_{2}} & = & mv'\left(\frac{v-1}{v'-1}-\frac{kv'-1}{k-1}\right)+\frac{v'(v'-k)}{k-1} +k^{2}e^{(1)}_{D_{1}D_{2}}+(k-1)e^{(2)}_{D_{1}D_{2}}+ke^{(3)}_{D_{1}D_{2}} \\
e^{(9)}_{D_{1}D_{2}} & = & m\left(\frac{v(v-1)}{v'(v'-1)}+v'^{2}-2v'\left(\frac{v-1}{v'-1}-\frac{v'-1}{k}\right)\right)-\frac{2(v'-1)(v'-k)}{k}- \\
 & & -(k^{2}-1)e^{(1)}_{D_{1}D_{2}}-(k-1)e^{(2)}_{D_{1}D_{2}}-(k-1)e^{(3)}_{D_{1}D_{2}}
\end{eqnarray*}
\normalsize
\end{proposition}

\begin{proof}
Since $D_{1},D_{2}$ are disjoint, this time the sums in \eqref{lemaj2} and \eqref{lemaj12} are zero; we compute the other sums in Lemma~\ref{lemaj} in the same way as in Propositions~\ref{prsyst1}-\ref{prsystk}. We add the following:
\begin{equation*}
m\frac{v(v-1)}{v'(v'-1)}-2 \ = \ \sum_{i=1}^{9}e^{(i)}_{D_{1}D_{2}}
\end{equation*}
to the list of equations, and inverting the linear system we obtain the result.
\end{proof}

For every $i$, we call $a^{(i)}$ (resp. $c^{(i)}$, $e^{(i)}$) the mean of all $a^{(i)}_{D_{1}D_{2}}$ (resp. $c^{(i)}_{D_{1}D_{2}}$, $e^{(i)}_{D_{1}D_{2}}$); in the notation of equations \eqref{lewdmbeasyk}-\eqref{lewdmbeasy1}-\eqref{lewdmbeasy0}, this signifies:
\begin{eqnarray}
\sum_{|D_{1}\cap D_{2}|=1}a^{(i)}_{D_{1}D_{2}} & = & nI_{1}a^{(i)} \label{meana} \\
\sum_{|D_{1}\cap D_{2}|=k}c^{(i)}_{D_{1}D_{2}} & = & nI_{k}c^{(i)} \label{meanc} \\
\sum_{|D_{1}\cap D_{2}|=0}e^{(i)}_{D_{1}D_{2}} & = & nI_{0}e^{(i)} \label{meane}
\end{eqnarray}
While computing the means, for every pair $(D_{1},D_{2})$ we count also $(D_{2},D_{1})$, so that we have some obvious equalities:
\begin{align}
a^{(3)}=& \ a^{(4)} & c^{(3)}=& \ c^{(4)} & e^{(2)}=& \ e^{(3)} \nonumber \\
a^{(5)}=& \ a^{(6)} & c^{(5)}=& \ c^{(6)} & e^{(4)}=& \ e^{(5)} \nonumber \\
a^{(7)}=& \ a^{(8)} & c^{(7)}=& \ c^{(8)} & e^{(7)}=& \ e^{(8)} \nonumber \\
a^{(10)}=& \ a^{(11)} & c^{(10)}=& \ c^{(11)} & & \label{equals}
\end{align}
From now on, while talking about the means we only consider the ones on the LHS of these equalities. In particular, notice that some of the means that correspond to the variables in Propositions~\ref{prsyst1}-\ref{prsystk}-\ref{prsyst0} are equal to each other; while considering the BIBD as a whole, as we are about to do, we need fewer free parameters.

\begin{proposition}\label{prsyst}
Let $V$ be a $(v,k,1)$-BIBD that has well-distributed minimal $(v',k,1)$-sub-BIBDs; suppose also that among pairs of such subdesigns there exist intersections of size $0$, $1$ and $k$. Then:
\small
\begin{eqnarray*}
a^{(5)} & = & \frac{(m-1)(v'-1)(v'-k)}{k(k-1)}-a^{(1)}-a^{(3)} \\
a^{(7)} & = & \frac{(m-1)(v'-1)}{k-1}-a^{(2)} \\
a^{(9)} & = & m(v'-1)^{2}-k^{2}a^{(1)}-(k-1)^{2}a^{(2)}-2ka^{(3)} \\
a^{(10)} & = & m\left(v-1-\frac{k(v'-1)^{2}}{k-1}\right)+\frac{(v'-1)(v'-k)}{k-1}+ \\
 & & + \ k^{2}a^{(1)}+(k-1)^{2}a^{(2)}+(2k-1)a^{(3)} \\
a^{(12)} & = & m\left(\frac{v-1}{v'-1}-2\frac{v'-1}{k-1}\right)+2\frac{v'-k}{k-1}+a^{(2)} \\
a^{(13)} & = & m\left(\frac{v(v-1)}{v'(v'-1)}+(v'-1)^{2}+\frac{v-1}{v'-1}-2v'\left(\frac{v-1}{v'-1}-\frac{v'-1}{k}\right)\right)- \\
 & & - \ \frac{2(v'-1)(v'-k)}{k}-(k^{2}-1)a^{(1)}-(k-1)^{2}a^{(2)}-2(k-1)a^{(3)} \\
c^{(1)} & = & \frac{v'-k}{k}\left(m\frac{v'-k^{2}}{k}+k-1\right)+\frac{I_{1}}{k^{2}I_{k}}\left(-2ka^{(1)}+(k-1)^{2}a^{(2)}-a^{(3)}\right) \\
c^{(2)} & = & \frac{k(m-1)(v'-k)}{k-1}-\frac{I_{1}}{I_{k}}a^{(2)} \\
c^{(3)} & = & \frac{I_{1}}{I_{k}}a^{(1)} \\
c^{(5)} & = & \frac{(m-k)(v'-k)^{2}}{k^{2}(k-1)}+\frac{I_{1}}{k^{2}I_{k}}\left(-k(k-2)a^{(1)}-(k-1)^{2}a^{(2)}+a^{(3)}\right) \\
c^{(7)} & = & \frac{I_{1}}{I_{k}}a^{(2)} \\
c^{(9)} & = & \frac{I_{1}}{I_{k}}a^{(3)} \\
c^{(10)} & = & I_{1}\left(\frac{v'-k}{v'}-\frac{1}{I_{k}}\left(a^{(1)}+a^{(3)}\right)\right) \\
c^{(12)} & = & I_{1}\left(\frac{k}{v'}-\frac{1}{I_{k}}a^{(2)}\right) \\
c^{(13)} & = & m\left(1-k+\frac{v(v-1)}{v'(v'-1)}+\frac{(v'-k)(v'-k^{2})}{k^{2}}+k\frac{v-1}{v'-1}-\right. \\
 & & \left.-2v'\left(\frac{v-1}{v'-1}-\frac{v'-1}{k}\right)\right)-\frac{(2v'-k-1)(v'-k)}{k} +\\
 & & +\frac{(k-1)I_{1}}{k^{2}I_{k}}\left(2ka^{(1)}+(k-1)a^{(2)}+(k+1)a^{(3)}\right) \\
e^{(1)} & = & \frac{(m-k)(v'-k)^{2}I_{k}}{k^{2}(k-1)I_{0}}+\frac{I_{1}}{k^{2}I_{0}}\left(-k(k-2)a^{(1)}-(k-1)^{2}a^{(2)}+a^{(3)}\right) \\
e^{(2)} & = & \frac{I_{1}}{I_{0}}\left(\frac{(m-1)(v'-1)(v'-k)}{k(k-1)}-a^{(1)}-a^{(3)}\right) \\
e^{(4)} & = & I_{k}\left(1-\frac{v'-k}{I_{0}}\left(m\left(\frac{v-1}{v'-1}-\frac{v'-1}{k-1}+\frac{v'-k}{k^{2}(k-1)}\right)+\frac{v'-k}{k}\right)\right)+ \\ & & + \ \frac{(k-1)I_{1}}{k^{2}I_{0}}\left(2ka^{(1)}+(k-1)a^{(2)}+(k+1)a^{(3)}\right) \\
e^{(6)} & = & mv'^{2}-\frac{k(v'-k)I_{k}}{I_{0}}\left(m\left(2\frac{v-1}{v'-1}-2\frac{v'-1}{k-1}+\frac{v'-k}{k(k-1)}\right)+\frac{v'-k}{k-1}\right)+ \\ & & + \ \frac{I_{1}}{I_{0}}\left(k^{2}a^{(1)}+(k-1)^{2}a^{(2)}+(2k-1)a^{(3)}\right) \\
e^{(7)} & = & \frac{(v'-k)I_{k}}{I_{0}}\left(m\left(\frac{(2k-1)(v-1)}{v'-1}-2v'+1\right)+v'-k\right)+mv'\left(\frac{v-1}{v'-1} \ - \right. \\ & & - \left. \frac{kv'-1}{k-1}\right)+\frac{v'(v'-k)}{k-1}-\frac{(k-1)I_{1}}{I_{0}}\left((k+1)a^{(1)}+(k-1)a^{(2)}+2a^{(3)}\right) \\
e^{(9)} & = & m\left(\frac{v(v-1)}{v'(v'-1)}+v'^{2}-2v'\left(\frac{v-1}{v'-1}-\frac{v'-1}{k}\right)\right)-\frac{2(v'-1)(v'-k)}{k}- \\ & & - \ \frac{(k-1)(v'-k)I_{k}}{I_{0}}\left(m\left(2\frac{v-1}{v'-1}-\frac{(2k+1)v'-k}{k^{2}}\right)+\frac{v'-k}{k}\right)+ \\ & & + \ \frac{(k-1)^{2}I_{1}}{k^{2}I_{0}}\left(k(k+2)a^{(1)}+(k^{2}-1)a^{(2)}+(2k+1)a^{(3)}\right)
\end{eqnarray*}
\normalsize
where $I_{k},I_{1},I_{0}$ are as in \eqref{lewdmbeasyk}-\eqref{lewdmbeasy1}-\eqref{lewdmbeasy0} respectively and do not depend on $a^{(1)},a^{(2)},a^{(3)}$.
\end{proposition}

\begin{proof}
We start with the statement of Propositions~\ref{prsyst1}-\ref{prsystk}-\ref{prsyst0}: since all the expressions are linear, the same relations hold for the means $a^{(i)},c^{(i)},e^{(i)}$. To these, we add some new equalities that relate the three systems to each other: in fact, notice that while computing means for every triple $(D_{1},D_{2},D_{3})$ with certain intersection sizes we count also all possible permutations of that triple in the appropriate permutations of such intersection sizes; for example, using \eqref{meana} and \eqref{meanc}:
\small
\begin{eqnarray*}
nI_{1}a^{(1)} & = & \sum_{|D_{1}\cap D_{2}|=1}a^{(1)}_{D_{1}D_{2}} \ = \ \left|\left\{(D_{1},D_{2},D_{3})\in\mathcal{D}^{3}\left|\begin{array}{l}|D_{1}\cap (D_{2}\setminus D_{3})|=1, \\ |D_{3}\cap (D_{1}\setminus D_{2})|=k, \\ |D_{3}\cap (D_{2}\setminus D_{1})|=k, \\ |D_{1}\cap D_{2}\cap D_{3}|=0 \end{array} \right.\right\}\right|=\\
 & = & \sum_{|D_{3}\cap D_{1}|=k}c^{(3)}_{D_{3}D_{1}} \ = \ nI_{k}c^{(3)}
\end{eqnarray*}
\normalsize
hence $I_{1}a^{(1)}=I_{k}c^{(3)}$. In general we obtain, using \eqref{meana}-\eqref{meanc}-\eqref{meane} and the same reasoning as above:
\begin{align*}
I_{1}a^{(1)}= & \ I_{k}c^{(3)} & I_{1}a^{(2)}= & \ I_{k}c^{(7)} & I_{1}a^{(3)}= & \ I_{k}c^{(9)} \\
I_{1}a^{(5)}= & \ I_{k}c^{(10)}=I_{0}e^{(2)} & I_{1}a^{(7)}= & \ I_{k}c^{(12)} & I_{1}a^{(10)}= & \ I_{0}e^{(6)} \\
I_{1}a^{(13)}= & \ I_{0}e^{(7)} & I_{k}c^{(5)}= & \ I_{0}e^{(1)} & I_{k}c^{(13)}= & \ I_{0}e^{(4)}
\end{align*}
Of course, some of these are redundant, as are the equalities in \eqref{equals}; we add as many as we need to the systems coming from the propositions and we obtain the result (since all three intersection sizes appear in the BIBD, the three quantities $I_{0},I_{1},I_{k}$ are all $>0$).
\end{proof}

It is clear that all the quantities in Proposition~\ref{prsyst} are $\geq 0$. There are also results that give upper bounds for some of them, including notably the three free variables $a^{(1)},a^{(2)},a^{(3)}$.

\begin{proposition}\label{prupper}
Let $V$ be a $(v,k,1)$-BIBD that has well-distributed minimal $(v',k,1)$-sub-BIBDs.

For any $D_{1},D_{2}\in\mathcal{D}$ with intersection size $1$:
\begin{eqnarray}
a^{(1)}_{D_{1}D_{2}} & \leq & \frac{(v'-1)(v'-k)}{k(k-1)}\left\lfloor\frac{v'-1}{k}\right\rfloor \label{pruppera1} \\
a^{(2)}_{D_{1}D_{2}} & \leq & \frac{(v'-1)^{2}}{(k-1)^{2}} \label{pruppera2} \\
a^{(3)}_{D_{1}D_{2}} & \leq & \frac{(v'-1)^{2}(v'-k)}{k(k-1)} \label{pruppera3}
\end{eqnarray}

For any $D_{1},D_{2}\in\mathcal{D}$ with intersection size $k$:
\begin{eqnarray}
c^{(1)}_{D_{1}D_{2}} & \leq & \frac{(v'-k^{2}+k-1)(v'-k)}{k(k-1)}\left\lfloor\frac{v'-k}{k}\right\rfloor \label{prupperc1} \\
c^{(2)}_{D_{1}D_{2}} & \leq & \frac{k(v'-k)^{2}}{(k-1)^{2}} \label{prupperc2} \\
c^{(3)}_{D_{1}D_{2}} & \leq & \frac{(v'-k^{2}+k-1)(v'-k)^{2}}{k(k-1)} \label{prupperc3}
\end{eqnarray}

For any $D_{1},D_{2}\in\mathcal{D}$ with intersection size $0$:
\begin{eqnarray}
e^{(1)}_{D_{1}D_{2}} & \leq & \frac{v'(v'-1)}{k(k-1)}\left\lfloor\frac{v'}{k}\right\rfloor \label{pruppere1} \\
e^{(2)}_{D_{1}D_{2}} & \leq & \frac{v'^{2}(v'-1)}{k(k-1)} \label{pruppere2}
\end{eqnarray}

In particular, the same bounds hold for their respective means $a^{(i)},c^{(i)},e^{(i)}$.
\end{proposition}

\begin{proof}
Fix $D_{1},D_{2}$ with $|D_{1}\cap D_{2}|=1$; for any given block $B\subseteq D_{1}$, the sub-BIBDs $D$ such that $D\cap D_{1}=B$ are all pairwise disjoint outside $B$: in particular, for a fixed $B$ that does not contain the vertex common to $D_{1}$ and $D_{2}$ (call it $x$), there are at most $\lfloor\frac{v'-1}{k}\rfloor$ such $D$ that are counted inside $a^{(1)}_{D_{1}D_{2}}$ (i.e., that also intersect $D_{2}$ in an entire block not containing $x$). By Lemma~\ref{lebibdeasy}\ref{lebibdeasy2}-\ref{lebibdeasy}\ref{lebibdeasy3}, the number of such $B$ is $\frac{v'-1}{k-1}\left(\frac{v'}{k}-1\right)$, and \eqref{pruppera1} follows.

For a fixed $B$ containing the unique $x\in D_{1}\cap D_{2}$, all $D\supseteq B$ are pairwise disjoint in $D_{2}\setminus\{x\}$; the number of possible $B$ is $\frac{v'-1}{k-1}$ by Lemma~\ref{lebibdeasy}\ref{lebibdeasy2}, and for each $B$ the number of possible intersections counted inside $a^{(2)}_{D_{1}D_{2}}$ is again $\frac{v'-1}{k-1}$, so we obtain \eqref{pruppera2}.

For \eqref{pruppera3} the reasoning is the same as for \eqref{pruppera1}, but now the intersections in $D_{2}$ are of size $1$, so there are $v'-1$ possibilities for each $B$.

Now we fix $D_{1},D_{2}$ with $|D_{1}\cap D_{2}|=k$. The number of blocks $B\subseteq D_{1}$ disjoint from the block $D_{1}\cap D_{2}$ (call it $B_{0}$) is given by \eqref{lewdmbeasyb0}; for each of them, the sub-BIBDs intersecting $D_{1}$ in $B$ and $D_{2}$ in another block are at most $\lfloor\frac{v'-k}{k}\rfloor$ because the intersections in $D_{2}$ must be all distinct. This proves \eqref{prupperc1}.

The number of blocks $B\subseteq D_{1}$ intersecting $B_{0}$ in one vertex is given by \eqref{lewdmbeasyb1}; for each of them, there are $\frac{v'-k}{k-1}$ blocks containing $B\cap B_{0}$ and distinct from $B_{0}$ (take Lemma~\ref{lebibdeasy}\ref{lebibdeasy2} and subtract $1$ to exclude $B_{0}$), and we get \eqref{prupperc2}.

For \eqref{prupperc3} the reasoning is analogous to \eqref{prupperc1}, with $v'-k$ possible intersections of size $1$ in $D_{2}$ for each $B$.

Finally we fix $D_{1},D_{2}$ disjoint. The number of blocks inside $D_{1}$ is given in Lemma~\ref{lebibdeasy}\ref{lebibdeasy3}; for any fixed $B\subseteq D_{1}$, there are at most $\lfloor\frac{v'}{k}\rfloor$ sub-BIBDs containing $B$ and intersecting $D_{2}$ in $k$ vertices and there are at most $v'$ sub-BIBDs containing $B$ and intersecting $D_{2}$ in $1$ vertex: this gives us \eqref{pruppere1} and \eqref{pruppere2}.
\end{proof}

The results in Proposition~\ref{prupper} are nontrivial because they involve counting sub-BIBDs that are constrained to intersect two other subdesigns, and at least one of them in $k$ vertices; this strongly limits possibilities, since sub-BIBDs cannot have more than $k$ vertices in common. To make a comparison, if we were to give a similar count for $a^{(7)}$, after fixing a block $B$ in $D_{1}$ the only constraint would be that only $m-1$ sub-BIBDs can intersect $D_{1}$ in $B$; this would lead to the bound $a^{(7)}\leq\frac{(m-1)(v'-1)}{k-1}$, which is obvious from Proposition~\ref{prsyst1}: indeed $\frac{(m-1)(v'-1)}{k-1}$ does not take into any consideration what happens to $D_{2}$, or, seen from another perspective, it corresponds to the trivial bound $a^{(7)}\leq a^{(2)}+a^{(7)}$.

\begin{corollary}\label{cobounds}
Let $V$ be a $(v,k,1)$-BIBD that has well-distributed minimal $(v',k,1)$-sub-BIBDs; suppose also that among pairs of such subdesigns there exist intersections of size $0$, $1$ and $k$. Then:
\begin{equation}\label{couppera1}
a^{(1)} \leq \min\left\{ \frac{(v'-1)(v'-k)}{k(k-1)}\left\lfloor\frac{v'-1}{k}\right\rfloor, \frac{I_{k}}{I_{1}}\frac{(v'-k^{2}+k-1)(v'-k)^{2}}{k(k-1)} \right\}
\end{equation}
\begin{equation}\label{colowera2}
a^{(2)} \geq \max\left\{ 0, \frac{I_{k}}{I_{1}}\frac{k(v'-k)}{k-1}\left(m-\frac{v'-1}{k-1}\right) \right\}
\end{equation}
\begin{equation}\label{couppera2}
a^{(2)} \leq \min\left\{\frac{(v'-1)^{2}}{(k-1)^{2}}, \frac{(m-1)(v'-1)}{k-1}, \frac{I_{k}}{I_{1}}\frac{k(v'-k)(m-1)}{k-1} \right\}
\end{equation}
\begin{equation}\label{colowera1a3}
a^{(1)}+a^{(3)} \geq \max\left\{ 0, \frac{v'-1}{k(k-1)}\left((m-1)(v'-k)-\frac{I_{0}}{I_{1}}v'^{2}\right) \right\}
\end{equation}
\begin{eqnarray}\label{couppera1a3}
a^{(1)}+a^{(3)} & \leq & \frac{(v'-1)(v'-k)}{k(k-1)}\min\left\{ \left\lfloor\frac{v'-1}{k}\right\rfloor+v'-1, \right.\nonumber \\
 & & \left. \frac{I_{k}}{I_{1}}\frac{(v'-k^{2}+k-1)(v'-k)}{v'-1}+v'-1,m-1 \right\}
\end{eqnarray}
\end{corollary}

\begin{proof}
In \eqref{couppera1}, the first term inside the RHS is just \eqref{pruppera1} while the second is obtained from \eqref{prupperc3} and the equality for $c^{(3)}$ in Proposition~\ref{prsyst}. Since every $a^{(2)}_{D_{1}D_{2}}$ is a nonnegative integer, $a^{(2)}$ is also nonnegative; the second term in \eqref{colowera2} comes from \eqref{prupperc2} and from $c^{(2)}$ in Proposition~\ref{prsyst}. The first term inside \eqref{couppera2} is simply \eqref{pruppera2}, the second and the third come from the nonnegativity of $a^{(7)},c^{(2)}$ and their expressions in Proposition~\ref{prsyst}. The sum $a^{(1)}+a^{(3)}$ is of course $\geq 0$ and the second lower bound is consequence of \eqref{pruppere2} and of the equality for $e^{(2)}$ in Proposition~\ref{prsyst}, thus giving us \eqref{colowera1a3}. Finally, the first two terms of \eqref{couppera1a3} come from summing \eqref{couppera1} and \eqref{pruppera3}, while the third comes from $a^{(5)}\geq 0$ and its expression in Proposition~\ref{prsyst}.
\end{proof}

Now we give some inequalities involving $m$. It is easy to already give a bound using only the fact that intersections of sub-BIBDs are of size at most $k$; knowing this, given a block $B\in\mathcal{B}$ all the subdesigns containing it are pairwise disjoint outside $B$, so that:
\begin{equation}\label{mvvp}
v\geq k+m(v'-k) \ \ \ \Longrightarrow \ \ \ m\leq\frac{v-k}{v'-k}
\end{equation}
We observe that the same bound appears also in \cite[Thm. 6]{DEF78} written in the form $n\leq\frac{v(v-1)(v-k)}{v'(v'-1)(v'-k)}$, which holds only for $v,n$ large enough but in more generality than just in Steiner systems like our $V$ (in turn, that theorem is an improvement of \cite[Thm. 3]{RCW75}).

It is possible to improve \eqref{mvvp} in some ranges. First, we give a result that allows us to naturally distinguish between a ``small'' $v$ and a ``large'' $v$ (with respect to $v'$; remember that by Lemma~\ref{lesubboundv} we already know that the ratio $\frac{v}{v'}$ is at least the order of magnitude of $k$).

\begin{lemma}\label{levsmalllarge}
Let $V$ be a $(v,k,1)$-BIBD with well-distributed minimal $(v',k,1)$-sub-BIBDs. Suppose that $k\geq 4$; then we have one of the two following possibilities:
\begin{enumerate}[(a)]
\item $v<\frac{1}{2}\left(1-\sqrt{1-\frac{4}{k}}\right)v'^{2}+\frac{1}{2}$ (i.e. $v$ is small);
\item $v>\frac{1}{2}\left(1+\sqrt{1-\frac{4}{k}}\right)v'^{2}+\frac{1}{2}$ (i.e. $v$ is large).
\end{enumerate}
\end{lemma}

\begin{proof}
We use the fact that $I_{0}$ as defined in \eqref{lewdmbeasy0} must be $\geq 0$. In particular we have:
\begin{equation*}
\frac{(v-1)(v-v'^{2})}{v'(v'-1)}+\frac{v'(v'-1)}{k}>0
\end{equation*}
We solve then $v^{2}-(v'^{2}+1)v+v'^{2}\left(1+\frac{(v'-1)^{2}}{k}\right)>0$; in the discriminant, we can comfortably forget about the lower degree terms (in $v'$) since by Proposition~\ref{prfisher} we have $v'\geq k^{2}-k+1$, and we obtain the result ($k\geq 4$ makes the discriminant positive).
\end{proof}

This means that $k\lesssim\frac{v}{v'}\lesssim\frac{v'}{k}$ when $v$ is small and $\frac{v}{v'}\gtrsim v'$ when $v$ is large (we use $\lesssim$ to indicate inequalities up to lower degree terms and/or for large values of the parameters, offering possibly a clearer intuition in exchange for a loss of precision in our assertions).

It turns out that when $v$ is small $m$ tends to be considerably smaller than the bound given in \eqref{mvvp}.

\begin{lemma}\label{leminvsmall}
Let $V$ be a $(v,k,1)$-BIBD with well-distributed minimal $(v',k,1)$-sub-BIBDs. Suppose that $v<\frac{(v'-1)^{2}}{k-1}+1$; then:
\begin{equation*}
m\leq\frac{\frac{v'-k}{k-1}}{\frac{v'-1}{k-1}-\frac{v-1}{v'-1}}
\end{equation*}
\end{lemma}

\begin{proof}
This is an immediate consequence of the nonnegativity of \eqref{lewdmbeasyx} and of the fact that with this condition on $v$ the coefficient of $m$ is negative.
\end{proof}

Actually, ``$v$ small'' as defined in Lemma~\ref{levsmalllarge} is slightly wider than the condition required in Lemma~\ref{leminvsmall}, but the overlap is almost complete (both have an extreme of magnitude $\frac{v'^{2}}{k}$). What is interesting is that in this case $m$ drops very fast with the decrease of $v$: already $v=\frac{(v'-1)^{2}}{2(k-1)}$ forces $m$ to be $1$ (it gives $m<2$, and $m$ is an integer).

\section{The graph isomorphism problem}

Balanced incomplete block designs with well-distributed minimal sub-BIBDs arise while examining a seemingly unrelated problem, the graph isomorphism problem (GIP): such BIBDs appear to be particularly resilient to the recent algorithm given by Babai \cite{Ba15} (see also \cite{HBD17}) that solves the problem in quasipolynomial time. We expose now some key features of the problem and the algorithm, while pointing out the connection to the structures described in the previous sections.

The GIP poses the simple question: given two generic graphs $\Gamma_{1},\Gamma_{2}$ on $v$ vertices, what is the time necessary to determine whether they are isomorphic and, in case of an affirmative answer, to describe the set of isomorphisms between the two graphs? A brute force algorithm has trivially factorial time in $v$, just by checking all possible permutations of vertices; since the 1980s we have had algorithms running in time $e^{O(\sqrt{v\log v})}$ (see \cite{BKL83}), and there have been other results holding in special cases (for example, graphs of bounded degree as in Luks \cite{Lu82}, working in polynomial time) or working for almost all graphs (see for example \cite{BES80}, in quadratic time).

A recent improvement by Babai \cite{Ba15} provides an algorithm that works deterministically for all graphs in quasipolynomial time, i.e. in time $e^{O(\log v)^{c}}$ for some constant $c$; a subsequent analysis by Helfgott \cite{HBD17} (see also the original version \cite{He17} in French) showed that it is possible to take $c=3$. It is still open and a target of investigation whether it is possible to produce an algorithm that solves the GIP in polynomial time, which would correspond to $c=1$; it is then reasonable to ask whether Babai's algorithm can be further improved in this direction, or at least refined to get $c=2$.

Let us examine very broadly the proof of Babai's claim; we will follow Helfgott's version of the result, as it is explicitly structured to obtain the condition $c=3$ that we want to improve. The proof works by recursion and breaks down graphs into smaller subgraphs in a suitable way in order to build the isomorphisms that we are looking for, or to exclude their existence; the process involves a certain degree of non-canonical choices, which translate to a multiplication of the runtime as they correspond in practice to a case-by-case solution of each possibility: the small amount of non-canonicity involved in the process is the key feature that brought down the runtime to quasipolynomial order in $v$, and any effort to further improve the result would likely involve an even greater reduction of the non-canonical choices that are necessary in the analysis.

A fundamental step in the main recursion that takes place in the algorithm is the Split-or-Johnson routine (see \cite[\S 7]{Ba15} \cite[\S 5.2]{HBD17}), which basically allows us to partition the graphs into subgraphs whose size is a fraction of the original one and to reduce the GIP for the whole graphs to several instances of the GIP for these smaller pieces: the process is not canonical, in fact its multiplicative cost is $e^{O(\log v)^{2}}$, and this fact combined with the recursion itself gives (among other things) the final cost with $c=3$ that we expect for the whole algorithm. In turn, the cost inside Split-or-Johnson is a consequence of the use of the Design Lemma (\cite[Thm. 6.1.1]{Ba15}, \cite[Prop. 5.1]{HBD17}): this lemma involves individualizing a handful of vertices, i.e. arbitrarily colouring them to make them distinct from all the others as defined in \cite[\S 1.3]{SW16}, and it is employed in this context to break the excessive symmetry that the graphs coming from $t$-designs exhibit and that makes them resilient to analysis; for the sake of clarity, we precise that the graph coming from a design $(V,\mathcal{B})$ is intended to be the bipartite graph having as set of vertices the set $V\cup\mathcal{B}$ and as edges the pairs $\{x,B\}$ such that $x\in B$. If there existed a way to deal with these kinds of graphs without resorting to the Design Lemma, it would already be a big step towards an improved algorithm with $c=2$.

Another tool at our disposal while dealing with the GIP is the Weisfeiler-Leman algorithm \cite{WL68}, a powerful tool that gives canonical colourings to graphs in polynomial time: it features prominently in many parts of Babai's proof, although it is not sufficient on its own to solve the GIP in reasonable time (see for example \cite{CFI92}). It is interesting to see whether applying this rather simple tool to our $t$-designs (more precisely, to the graphs coming from them) is powerful enough to break them, either on its own or coupled with a small individualization; indeed, small computational experiments show that very often Weisfeiler-Leman is enough to give a canonical colouring of the vertices and blocks of BIBDs, with an evident exception.

\begin{proposition}\label{prwllambda1}
Let $\Gamma$ be a bipartite graph with set of vertices $V\cup\mathcal{B}$ corresponding to a BIBD $(V,\mathcal{B})$. Suppose that $\lambda=1$: then, applying the Weisfeiler-Leman algorithm to $\Gamma$ yields a colouring on the vertices that at best distinguishes $V$ from $\mathcal{B}$, while leaving $V$ and $\mathcal{B}$ themselves monochromatic on the vertices.
\end{proposition}

\begin{proof}
The Weisfeiler-Leman algorithm iterates the following refinement procedure: starting from a colouring of $\Gamma^{2}$, for every $(x,y)$ we add to its colour the information about the colours of all triangles $(x,z,y)$ built on it; when the refinement step does not create any new colour, the algorithm stops (for a precise definition, see for instance \cite[Alg. 3]{HBD17}). The starting point is the colouring $\mathcal{C}_{0}$ consisting of the three colours ``vertex'', ``edge'', ``not edge'' assigned in the obvious way to the pairs of vertices $(x,y)\in\Gamma^{2}$; we claim that the algorithm does not manage to further refine the colouring $\mathcal{C}$ given by:
\begin{eqnarray*}
(x,y)\in\text{``vertex''} & \Longleftrightarrow & x=y\in V \\
(x,y)\in\text{``block''} & \Longleftrightarrow & x=y\in\mathcal{B} \\
(x,y)\in\text{``belongs''} & \Longleftrightarrow & x\in V,y\in\mathcal{B},x\in y \\
(x,y)\in\text{``doesn't belong''} & \Longleftrightarrow & x\in V,y\in\mathcal{B},x\not\in y \\
(x,y)\in\text{``contains''} & \Longleftrightarrow & x\in\mathcal{B},y\in V,y\in x \\
(x,y)\in\text{``doesn't contain''} & \Longleftrightarrow & x\in\mathcal{B},y\in V,y\not\in x \\
(x,y)\in\text{``vertices''} & \Longleftrightarrow & x\in V,y\in V,x\neq y \\
(x,y)\in\text{``0'',``1''} & \Longleftrightarrow & x\in\mathcal{B},y\in\mathcal{B},|x\cap y|=0,1
\end{eqnarray*}
If the claim is true, in particular it means that $\mathcal{C}$ is a refinement of $\mathcal{C}_{0}$ that would make the algorithm stop and that distinguishes $V$ from $\mathcal{B}$ (coloured ``vertex'' and ``block'' respectively) but does not distinguish anything inside each of them: this proves the proposition. To be clear, there could be BIBDs that do not even reach this $\mathcal{C}$, namely symmetric BIBDs, as shown in \cite[Ex. B.13]{HBD17}.

Proving the claim is just a dull exercise of verifying that the number of $z$ such that $(x,z),(z,y)$ have a given colour is independent from the choice of $x,y$ (among all $(x,y)$ having a certain fixed colour); everything reduces to the use of $v,k,b,r$ (in particular $r$ is independent from the vertex it refers to, by Lemma~\ref{lebibdeasy}\ref{lebibdeasy1}) and to the fact that for any two vertices in $V$ there exists a unique block containing both, since $\lambda=1$, and for any two blocks there exists at most one vertex in their intersection, again since $\lambda=1$. For the sake of brevity, we call the $9$ colours described above respectively: $\text{vx},\text{bl},\text{be},\text{dbe},\text{co},\text{dco},\text{vs},\text{0},\text{1}$.

For $(x,x)\in\text{vx}$, there is one vertex $z$ ($z=x$) such that $(x,z),(z,x)\in\text{vx}$; there are $v-1$ vertices ($z\in V\setminus\{x\}$) such that $(x,z),(z,x)\in\text{vs}$; there are $r$ vertices ($z\in\mathcal{B}$ with $x\in z$) such that $(x,z)\in\text{be},(z,x)\in\text{co}$; and there are $b-r$ vertices ($z\in\mathcal{B}$ with $x\not\in z$) such that $(x,z)\in\text{dbe},(z,x)\in\text{dco}$.

For $(x,x)\in\text{bl}$, there is one vertex ($z=x$) such that $(x,z),(z,x)\in\text{bl}$; there are $\frac{k(v-k)}{k-1}$ vertices ($z\in\mathcal{B}$ with $|z\cap x|=1$ by \eqref{lewdmbeasyb1}) such that $(x,z),(z,x)\in\text{1}$; there are $\frac{(v-k^{2}+k-1)(v-k)}{k(k-1)}$ vertices ($z\in\mathcal{B}$ with $|z\cap x|=0$ by \eqref{lewdmbeasyb0}) such that $(x,z),(z,x)\in\text{0}$; there are $k$ vertices ($z\in V$ with $z\in x$) such that $(x,z)\in\text{co},(z,x)\in\text{be}$; and there are $v-k$ vertices ($z\in V$ with $z\not\in x$) such that $(x,z)\in\text{dco},(z,x)\in\text{dbe}$.

For $(x,y)\in\text{be}$, there is one vertex ($z=x$) such that $(x,z)\in\text{vx},(z,y)\in\text{be}$; there are $k-1$ vertices ($z\in y\setminus\{x\}$) such that $(x,z)\in\text{vs},(z,y)\in\text{be}$; there are $v-k$ vertices ($z\in V\setminus y$) such that $(x,z)\in\text{vs},(z,y)\in\text{dbe}$; there is one vertex ($z=y$) such that $(x,z)\in\text{be},(z,y)\in\text{bl}$; there are $r-1$ vertices ($z\in\mathcal{B}\setminus\{y\}$ with $x\in z$) such that $(x,z)\in\text{be},(z,y)\in\text{1}$; there are $\frac{k(v-k)}{k-1}-r+1$ vertices ($z\in\mathcal{B}$ intersecting $y$ outside $x$ by \eqref{lewdmbeasyb1}) such that $(x,z)\in\text{dbe},(z,y)\in\text{1}$; and there are $\frac{(v-k^{2}+k-1)(v-k)}{k(k-1)}$ vertices ($z\in\mathcal{B}$ with $|z\cap y|=0$ by \eqref{lewdmbeasyb0}) such that $(x,z)\in\text{dbe},(z,y)\in\text{0}$.

For $(x,y)\in\text{dbe}$, there is one vertex ($z=x$) such that $(x,z)\in\text{vx},(z,y)\in\text{dbe}$; there are $k$ vertices ($z\in y$) such that $(x,z)\in\text{vs},(z,y)\in\text{be}$; there are $v-k-1$ vertices ($z\in V\setminus (y\cup\{x\})$) such that $(x,z)\in\text{vs},(z,y)\in\text{dbe}$; there is one vertex ($z=y$) such that $(x,z)\in\text{dbe},(z,y)\in\text{bl}$; there are $k$ vertices ($z\in\mathcal{B}$ with $x\in z$ and $|z\cap y|=1$) such that $(x,z)\in\text{be},(z,y)\in\text{1}$; there are $r-k$ vertices ($z\in\mathcal{B}$ with $x\in z$ and $|z\cap y|=0$) such that $(x,z)\in\text{be},(z,y)\in\text{0}$; there are $\frac{k(v-k)}{k-1}-k$ vertices ($z\in\mathcal{B}$ with $x\not\in z$ and $|z\cap y|=1$ by \eqref{lewdmbeasyb1}) such that $(x,z)\in\text{dbe},(z,y)\in\text{1}$; and there are $\frac{(v-k^{2}+k-1)(v-k)}{k(k-1)}-r+k$ vertices ($z\in\mathcal{B}$ with $|z\cap(y\cup\{x\})|=0$ by \eqref{lewdmbeasyb0}) such that $(x,z)\in\text{dbe},(z,y)\in\text{0}$.

For $(x,y)\in\text{co}$, there is one vertex ($z=x$) such that $(x,z)\in\text{bl},(z,y)\in\text{co}$; there are $r-1$ vertices ($z\in\mathcal{B}\setminus\{x\}$ with $y\in z$) such that $(x,z)\in\text{1},(z,y)\in\text{co}$; there are $\frac{k(v-k)}{k-1}-r+1$ vertices ($z\in\mathcal{B}$ intersecting $x$ outside $y$ by \eqref{lewdmbeasyb1}) such that $(x,z)\in\text{1},(z,y)\in\text{dco}$; there are $\frac{(v-k^{2}+k-1)(v-k)}{k(k-1)}$ vertices ($z\in\mathcal{B}$ with $|z\cap x|=0$ by \eqref{lewdmbeasyb0}) such that $(x,z)\in\text{0},(z,y)\in\text{dco}$; there is one vertex ($z=y$) such that $(x,z)\in\text{co},(z,y)\in\text{vx}$; there are $k-1$ vertices ($z\in x\setminus\{y\}$) such that $(x,z)\in\text{co},(z,y)\in\text{vs}$; and there are $v-k$ vertices ($z\in V\setminus x$) such that $(x,z)\in\text{dco},(z,y)\in\text{vs}$.

For $(x,y)\in\text{dco}$, there is one vertex ($z=x$) such that $(x,z)\in\text{bl},(z,y)\in\text{dco}$; there are $k$ vertices ($z\in\mathcal{B}$ with $|z\cap x|=1$ and $y\in z$) such that $(x,z)\in\text{1},(z,y)\in\text{co}$; there are $r-k$ vertices ($z\in\mathcal{B}$ with $|z\cap x|=0$ and $y\in z$) such that $(x,z)\in\text{0},(z,y)\in\text{co}$; there are $\frac{k(v-k)}{k-1}-k$ vertices ($z\in\mathcal{B}$ with $|z\cap x|=1$ and $y\not\in z$ by \eqref{lewdmbeasyb1}) such that $(x,z)\in\text{1},(z,y)\in\text{dco}$; there are $\frac{(v-k^{2}+k-1)(v-k)}{k(k-1)}-r+k$ vertices ($z\in\mathcal{B}$ with $|z\cap(x\cup\{y\})|=0$ by \eqref{lewdmbeasyb0}) such that $(x,z)\in\text{0},(z,y)\in\text{dco}$; there is one vertex ($z=y$) such that $(x,z)\in\text{dco},(z,y)\in\text{vx}$; there are $k$ vertices ($z\in x$) such that $(x,z)\in\text{co},(z,y)\in\text{vs}$; and there are $v-k-1$ vertices ($z\in V\setminus(x\cup\{y\})$) such that $(x,z)\in\text{dco},(z,y)\in\text{vs}$.

For $(x,y)\in\text{vs}$, there is one vertex ($z=x$) such that $(x,z)\in\text{vx},(z,y)\in\text{vs}$; there is one vertex ($z=y$) such that $(x,z)\in\text{vs},(z,y)\in\text{vx}$; there are $v-2$ vertices ($z\in V\setminus\{x,y\}$) such that $(x,z),(z,y)\in\text{vs}$; there is one vertex (the unique $z\in\mathcal{B}$ with $x,y\in z$) such that $(x,z)\in\text{be},(z,y)\in\text{co}$; there are $r-1$ vertices ($z\in\mathcal{B}$ with $x\in z$ and $y\not\in z$) such that $(x,z)\in\text{be},(z,y)\in\text{dco}$; there are $r-1$ vertices ($z\in\mathcal{B}$ with $x\not\in z$ and $y\in z$) such that $(x,z)\in\text{dbe},(z,y)\in\text{co}$; and there are $b-2r+1$ vertices ($z\in\mathcal{B}$ with $x,y\not\in z$) such that $(x,z)\in\text{dbe},(z,y)\in\text{dco}$.

For $(x,y)\in\text{0}$, there is one vertex ($z=x$) such that $(x,z)\in\text{bl},(z,y)\in\text{0}$; there is one vertex ($z=y$) such that $(x,z)\in\text{0},(z,y)\in\text{bl}$; there are $k^{2}$ vertices ($z\in\mathcal{B}$ with $|z\cap x|=|z\cap y|=1$) such that $(x,z),(z,y)\in\text{1}$; there are $\frac{k(v-k)}{k-1}-k^{2}$ vertices ($z\in\mathcal{B}$ with $|z\cap x|=1$ and $|z\cap y|=0$ by \eqref{lewdmbeasyb1}) such that $(x,z)\in\text{1},(z,y)\in\text{0}$; there are $\frac{k(v-k)}{k-1}-k^{2}$ vertices ($z\in\mathcal{B}$ with $|z\cap x|=0$ and $|z\cap y|=1$ by \eqref{lewdmbeasyb1}) such that $(x,z)\in\text{0},(z,y)\in\text{1}$; there are $b-2-2\frac{k(v-k)}{k-1}+k^{2}$ vertices ($z\in\mathcal{B}$ with $|z\cap x|=|z\cap y|=0$ by \eqref{lewdmbeasyb1}) such that $(x,z),(z,y)\in\text{0}$; there are $k$ vertices ($z\in x$) such that $(x,z)\in\text{co},(z,y)\in\text{dbe}$; there are $k$ vertices ($z\in y$) such that $(x,z)\in\text{dco},(z,y)\in\text{be}$; and there are $v-2k$ vertices ($z\not\in x\cup y$) such that $(x,z)\in\text{dco},(z,y)\in\text{dbe}$.

For $(x,y)\in\text{1}$, there is one vertex ($z=x$) such that $(x,z)\in\text{bl},(z,y)\in\text{1}$; there is one vertex ($z=y$) such that $(x,z)\in\text{1},(z,y)\in\text{bl}$; there are $r-2+(k-1)^{2}$ vertices ($z\in\mathcal{B}$ with $|z\cap x|=|z\cap y|=1$ by summing the blocks containing the common vertex and the blocks bridging $x,y$ through two other vertices) such that $(x,z),(z,y)\in\text{1}$; there are $\frac{k(v-k)}{k-1}-r+1-(k-1)^{2}$ vertices ($z\in\mathcal{B}$ with $|z\cap x|=1$ and $|z\cap y|=0$ by what we said before and \eqref{lewdmbeasyb1}) such that $(x,z)\in\text{1},(z,y)\in\text{0}$; there are $\frac{k(v-k)}{k-1}-r+1-(k-1)^{2}$ vertices ($z\in\mathcal{B}$ with $|z\cap x|=0$ and $|z\cap y|=1$ as before) such that $(x,z)\in\text{0},(z,y)\in\text{1}$; there are $\frac{(v-k^{2}+k-1)(v-k)}{k(k-1)}-\frac{k(v-k)}{k-1}+r-1+(k-1)^{2}$ vertices ($z\in\mathcal{B}$ with $|z\cap x|=|z\cap y|=0$ by \eqref{lewdmbeasyb0} and what we said before) such that $(x,z),(z,y)\in\text{0}$; there is one vertex (the unique $z\in V$ with $z\in x\cap y$) such that $(x,z)\in\text{co},(z,y)\in\text{be}$; there are $k-1$ vertices ($z\in x\setminus y$) such that $(x,z)\in\text{co},(z,y)\in\text{dbe}$; there are $k-1$ vertices ($z\in y\setminus x$) such that $(x,z)\in\text{dco},(z,y)\in\text{be}$; and there are $v-2k+1$ vertices ($z\not\in x\cup y$) such that $(x,z)\in\text{dco},(z,y)\in\text{dbe}$.
\end{proof}

What this proposition shows is that, if we want to crack BIBDs with $\lambda=1$ using Weisfeiler-Leman, individualization is a necessity; the problem is that this defies the whole point of avoiding the use of the Design Lemma, because it introduces again an element of non-canonicity that multiplies the cost. In order to escape the issue, there are two things that can play in our favour: we can strive to get better results than just breaking graphs into pieces of fractional size, or we can use canonical techniques other than (and undetected by) Weisfeiler-Leman. An example of how realizing the former is an actual possibility is the following result.

\begin{proposition}\label{prsteiner3}
Let $(V,\mathcal{B})$ be a $(v,k,1)$-BIBD. Suppose that $k=3$: then there is a set of at most $1+\lceil\log_{2}v\rceil$ vertices such that, after individualizing them and applying Weisfeiler-Leman, the graph corresponding to $(V,\mathcal{B})$ has a final colouring that gives to every pair of vertices a different colour (we say that the set completely splits the graph, as defined in \cite[\S 1.3]{SW16}).
\end{proposition}

\begin{proof}
We start individualizing any $3$ vertices $x_{1},x_{2},x_{3}\in V$ not in the same block (i.e. not adjacent to the same vertex $y\in\mathcal{B}$). We establish some useful facts.

First, if two vertices $x_{1},x_{2}$ have a distinct colour then the block containing them acquires a distinct colour too: in fact, for any $x$ with a distinct colour $c(x)$, all the edges that have $x$ as an endpoint get a refined colouring containing the information ``its endpoint has colour $c(x)$''; more precisely, $(y,z)$ of colour $c(y,z)$ has endpoint $x$ (i.e. $z=x$) if and only if the number of vertices $w$ such that the colour of $(y,w)$ is $c(y,z)$ and the colour of $(w,z)$ is $c(x)$ is $1$ (namely, it would be the vertex $w=x$ itself). Now, having these two vertices $x_{1},x_{2}$ and having coloured the edges incident to them, the unique block containing the two vertices is also the unique vertex adjacent to both of them in the graph: in other words, it is the only vertex that is the starting point of an edge whose colour knows that ``its endpoint has colour $c(x_{1})$'' and an edge whose colour knows that ``its endpoint has colour $c(x_{2})$''; at the next iteration of the algorithm, the block acquires a distinct colour that contains this entire information.

Second, if a block $B$ has a distinct colour then all of its vertices have a colour that knows that they belong to $B$: in fact, every edge starting from $B$ acquires a colour that has this information, and then the endpoints do too. Combining this with the previous observation, if two vertices $x_{1},x_{2}$ have distinct colours then the unique third vertex that constitutes the block acquires also a distinct colour.

Now, we are starting by individualizing three vertices: let us say that we have an ordering on these distinct colours, for simplicity we say that $x_{1}$ has colour ``1'', $x_{2}$ has colour ``2'' and $x_{3}$ has colour ``3''. To understand what the Weisfeiler-Leman algorithm is doing, we iterate the following procedure: given a set of vertices $S$ with distinct ordered colours, we consider all pairs of vertices $\{x_{i},x_{j}\}\subseteq S$ that belong to a block with the third vertex outside $S$; by the reasoning above, all these third vertices also acquire a distinct colour, and we induce on these new colours an ordering defined lexicographically on the pairs: for every third vertex $x_{k}$ and every $x_{i}\in S$ we have $c(x_{k})>c(x_{i})$, while for every two third vertices $x_{k},y_{k}$ coming from $\{x_{i},x_{j}\},\{y_{i},y_{j}\}\subseteq S$ we have that $c(x_{k})>c(y_{k})$ if and only if either $c(x_{i})>c(y_{i})$ or both $c(x_{i})=c(y_{i})$ (which means $x_{i}=y_{i}$) and $c(x_{j})>c(y_{j})$. After having ordered all these third colours, we add the third vertices to $S$ and we repeat. For example at the first step, since $x_{1},x_{2},x_{3}$ are not in the same block, there are three new vertices $x_{4},x_{5},x_{6}$ coming from blocks $\{x_{1},x_{2},x_{4}\},\{x_{1},x_{3},x_{5}\},\{x_{2},x_{3},x_{6}\}$: since $(1,2)<(1,3)<(2,3)$ we obtain that $x_{4},x_{5},x_{6}$ are coloured ``4'', ``5'', ``6'' respectively.

Whichever acquisition of distinct colour happens through the procedure described above, it happens also through Weisfeiler-Leman because of the facts that we established above; therefore, we have only to ensure that the procedure distinguishes all vertices in $V$, and we would be done (if we have all distinct colours in $V$, obviously we have also distinct colours in $\mathcal{B}$: every $B\in\mathcal{B}$ is adjacent to a distinct triple of vertices, which means a distinct triple of colours, thus acquiring itself a distinct colour).

The procedure however may stop before covering the whole $V$. This happens only if at an intermediate step the set $S$ has no pairs $\{x_{i},x_{j}\}$ with a third vertex $x_{k}$ outside $S$, i.e. only if $S$ contains all blocks of $(V,\mathcal{B})$ covering all pairs of $S$: in other words, only if $S$ is a sub-BIBD. When this happens, we individualize a new vertex $y\not\in S$, we add it to $S$ and we repeat the procedure again; we go on until the whole $V$ is covered.

How many vertices have been individualized in order to make the procedure work until the end? We have started with $3$ vertices and we individualize as many other vertices as the length of a chain of sub-BIBDs nested into each other, from the minimal sub-BIBD containing the initial $3$ vertices up to the whole $V$; by Lemma~\ref{lesubboundv}, for each two consecutive sub-BIBDs $V_{1}\subseteq V_{2}$ in the chain, we have $|V_{2}|\geq (k-1)|V_{1}|+1>2|V_{1}|$. Therefore the number of vertices that is sufficient to individualize is $<3+\log_{2}\frac{v}{3}<2+\log_{2}v$.
\end{proof}

What this proposition shows is that after individualizing $O(\log v)$ vertices we are able to conclude our search for isomorphisms just by testing one possibility: since all vertices have distinct colours, there is only one possible bijection between the two graphs preserving colours, and for that bijection we have to check whether edges in one graph correspond to edges in the other graph; for a Steiner triple system, there are $3b=\frac{v(v-1)}{2}$ edges in the graph, so the check is performed in polynomial time. Accounting for the multiplication cost, the GIP for Steiner triple systems is solved in time $e^{O(\log v)^{2}}$. We also observe that the proposition implies that Steiner triple systems have an automorphism group of size at most $e^{O(\log v)^{2}}$; this is of course neither unknown nor unexpected, as there exist several results classifying automorphism groups of many subclasses of Steiner triple systems and we actually know that with extremely high probability a Steiner triple system has no nontrivial automorphisms (see \cite{Ba80}).

Another fact that is hinted at in the proof of Proposition~\ref{prsteiner3} and implied by Proposition~\ref{prwllambda1} is that sub-BIBDs are not detected by Weisfeiler-Leman, in the sense that in a graph deriving from a $(v,k,1)$-BIBD $(V,\mathcal{B})$ two vertices corresponding to two elements $x,y\in V$ or $x,y\in\mathcal{B}$, one of which contained in a sub-BIBD and one of which not contained in any of them, do not receive different colours after applying Weisfeiler-Leman (by Proposition~\ref{prwllambda1}, no two such vertices are distinguished in any way); on the other hand, sub-BIBDs are canonical, in the sense that they are preserved by isomorphisms: what is more, other data like the number of sub-BIBDs and the size of sub-BIBDs containing a given vertex or block are clearly preserved as well. Finally, despite being out of reach for a pure Weisfeiler-Leman, minimal sub-BIBDs are fast to detect.

\begin{lemma}\label{letimesub}
Let $(V,\mathcal{B})$ be a $(v,k,1)$-BIBD. Then in time $O(v^{5})$ we can produce the list of all minimal sub-BIBDs, and in particular for every $x\in V$ and $B\in\mathcal{B}$ we can obtain the number of minimal sub-BIBDs containing them.
\end{lemma}

\begin{proof}
First, we compile the following list: for every pair of vertices in $V$, we assign the block of $\mathcal{B}$ containing the pair (unique by $\lambda=1$). For every $B\in\mathcal{B}$, we just have to assign $B$ to the $\binom{k}{2}$ pairs inside $B$: this takes time $O\left(b\binom{k}{2}\right)=O(v^{2})$ by Lemma~\ref{lebibdeasy}\ref{lebibdeasy3}.

Take any triple $\{x_{1},x_{2},x_{3}\}\subseteq V$ of vertices not all in the same block. We can find what is the minimal sub-BIBD that contains these three vertices following the procedure described in Proposition~\ref{prsteiner3}: we set $S=\{x_{1},x_{2},x_{3}\}$, for every pair in $S$ we add the unique block that contains it to $S$, and we repeat the procedure until only pairs with blocks already inside $S$ are left. Having already created the aforementioned list, it takes time $O(v^{2})$ again (the number of pairs that we have to check inside $S\subseteq V$) to produce the sub-BIBD containing $x_{1},x_{2},x_{3}$ and the list of vertices and blocks contained in it.

There are $\binom{v}{3}$ triples $\{x_{1},x_{2},x_{3}\}\subseteq V$; we can avoid the triples that are contained in one block, and we can be careful about double-counting of sub-BIBDs, but this leads to improvements up to constants at most. Therefore, we have to repeat the process above $O(v^{3})$ times, which gives the result. The computation of the numbers of minimal sub-BIBDs per vertex/block is a byproduct of our algorithm, easy to account for.
\end{proof}

Thanks to it, we have the interesting possibility of focusing on the sub-BIBDs themselves instead of the blocks, when they exist: this can be an advantage in many situations, thanks to the canonicity and rapidity of the whole procedure.

\begin{proposition}
Let $(V,\mathcal{B})$ be a $(v,k,1)$-BIBD and let $\Gamma$ be the corresponding graph with set of vertices $V\cup\mathcal{B}$. Then we have one of these possibilities:
\begin{enumerate}[(a)]
\item $V$ has no sub-BIBDs;
\item if $V$ has sub-BIBDs but does not satisfy Definition~\ref{dewdmb}, it is possible to find a canonical colouring of $\Gamma$, nontrivial on $\mathcal{B}$, in polynomial time;
\item if $V$ satisfies Definition~\ref{dewdmb} with $m=1$, it is possible to find a canonical nontrivial partition of $\mathcal{B}$ in polynomial time;
\item if $V$ satisfies Definition~\ref{dewdmb}, it is possible to find a graph $\Gamma'$ with set of vertices $V\cup\mathcal{D}$ and such that, if we know $\text{\normalfont Aut}(\Gamma')$, we can obtain $\text{\normalfont Aut}(\Gamma)$ in polynomial time.
\end{enumerate}
\end{proposition}

\begin{proof}
We perform the procedure described in Lemma~\ref{letimesub}. If the minimal sub-BIBD is $V$ for all examined triples, we are in the first case. If the minimal sub-BIBDs have size $v'<v$ but do not cover $V$ or $\mathcal{B}$ evenly, it will not cover $\mathcal{B}$ evenly in particular (condition~(\ref{dewdmb2}) in Definition~\ref{dewdmb} would imply condition~(\ref{dewdmb1}): observe the proof of Proposition~\ref{prwdmbbibd}\ref{prwdmbbibd1} and keep in mind Lemma~\ref{lebibdeasy}\ref{lebibdeasy1}); using Lemma~\ref{letimesub} and assigning to each $B\in\mathcal{B}$ a colour that expresses how many sub-BIBDs contain $B$, we get a colouring on $\mathcal{B}$ that is canonical, nontrivial and polynomial-time. If $(V,\mathcal{B})$ has well-distributed minimal sub-BIBDs with $m=1$, it means in particular that every block belongs to exactly one minimal subdesign; then by Lemma~\ref{letimesub} we are able to find them in polynomial time, and $\mathcal{D}$ (seen as a set of subsets of $\mathcal{B}$) is a canonical partition of $\mathcal{B}$.

Finally, suppose that $(V,\mathcal{B})$ has well-distributed minimal sub-BIBDs and that $m>1$. Now, by Proposition~\ref{prwdmbbibd}\ref{prwdmbbibd1} $(V,\mathcal{D})$ is a BIBD as well, so we choose $\Gamma'$ to be the corresponding graph with set of vertices $V\cup\mathcal{D}$: to build the graph we run Lemma~\ref{letimesub}, so that we obtain the collection $\mathcal{D}$ of subdesigns and the set of vertices and blocks that each $D\in\mathcal{D}$ contains. It is obvious that any automorphism of $\Gamma$ is also an automorphism of $\Gamma'$ and vice versa, by being induced by the same bijection on the set $V$: $\Gamma$ has blocks that are more refined than the subdesigns of $\Gamma'$, while on the other hand every $B\in\mathcal{B}$ is definable as intersection of some $D_{1},D_{2}\in\mathcal{D}$ (since $m>1$) so that every automorphism of $\Gamma'$ preserves them too.

Suppose now that we are given a set of generators of $\text{Aut}(\Gamma')$: we can easily translate them into automorphisms of $\Gamma$, and then we are done because they would generate $\text{Aut}(\Gamma)$ as well. As we already said, each $\varphi\in\text{Aut}(\Gamma')$ is induced by a bijection on the set $V$, so the corresponding automorphism in $\Gamma$ can be retrieved by looking at the action of $\varphi$ on $V$ alone.
\end{proof}

This result explains the interest in sub-BIBDs in the context of the GIP. When we have no sub-BIBDs, individualizing three vertices and applying Weisfeiler-Leman allows us to give colours to all the vertices of the graph, without being interrupted by the borders of a subdesign as in Proposition~\ref{prsteiner3}; when $k>3$ the vertices are not uniquely coloured, so in principle we have not solved the problem yet (we have made non-canonical choices obtaining at best the same fractional reduction that the Design Lemma already guaranteed), but it is worth noting that computational experiments show that Weisfeiler-Leman does indeed break down the whole graph in many cases. In the second case of the proposition above, we are already able to reduce the GIP to the smaller subproblems coming from the different colours in which $\mathcal{B}$ is divided and, everything being canonical and polynomial-time, we are in the best possible situation. The third case features a partition of $\mathcal{B}$ into parts of equal size, each one corresponding to a $D\in\mathcal{D}$: this is another way in which the recursion in Babai's algorithm works.

The fourth case tells us that studying the design $(V,\mathcal{D})$ is basically equivalent to study the original design $(V,\mathcal{B})$, since we can translate easily automorphisms of one into automorphisms of the other. $(V,\mathcal{D})$ has advantages of its own: having $\lambda=m$, it escapes the barrier posed by Proposition~\ref{prwllambda1}. What is more, Weisfeiler-Leman is strong enough to be able to exploit some asymmetries that are easy to describe and that is worth investigating; for example, if we suppose that $|D_{1}\cap D_{2}|=|D'_{1}\cap D'_{2}|=k$ but $c^{(1)}_{D_{1}D_{2}}\neq c^{(1)}_{D'_{1}D'_{2}}$, we know by the mere definition of the algorithm that the pairs $(D_{1},D_{2}),(D'_{1},D'_{2})$ will receive distinct colours: after all we are just counting how many triangles $(D_{1},D,D_{2})$ and $(D'_{1},D,D'_{2})$ we can build with all sides of colour ``k'', in the language of Proposition~\ref{prwllambda1}, and if the two numbers are different then the colours will be different too (obviously there is nothing special about $c^{(1)}$, any other one of the parameters $a^{(i)},c^{(i)},e^{(i)}$ will do). Finally, if we pass to $(V,\mathcal{D})$ we are now working with a graph of size $v+n$ instead of $v+b$: if $n<b$, this already counts as a reduction to a smaller problem, a reduction that again happens to be canonical and polynomial-time; the BIBDs in Examples~\ref{extheone3}-\ref{extheone4} for instance have $n=15<35=b$ and $n=40<130=b$, respectively. They are not the only case in which this happens, as we show now.

\begin{proposition}\label{prspecm}
Let $(V,\mathcal{B})$ be a $(v,k,1)$-BIBD that has well-distributed minimal $(v',k,1)$-sub-BIBDs; suppose that there are no intersections of two minimal subdesigns of size $1$. Then $n<b$.

Let $(V,\mathcal{B})$ be a $(v,k,1)$-BIBD that has well-distributed minimal $(v',k,1)$-sub-BIBDs; suppose that $v'>k^{2}-k+1$ (i.e. the sub-BIBDs are not symmetric) and that there are no intersections of two minimal subdesigns of size $0$. Then $n<b$.
\end{proposition}

\begin{proof}
By \eqref{lewdmbeasy1}, having no minimal sub-BIBD intersection of size $1$ means that:
\begin{equation*}
v'\left(m\left(\frac{v-1}{v'-1}-\frac{v'-1}{k-1}\right)+\frac{v'-k}{k-1}\right)=0
\end{equation*}
In particular, the coefficient of $m$ must be negative, which implies:
\begin{equation*}
v<\frac{(v'-1)^{2}}{k-1}+1
\end{equation*}
Using \eqref{mvvp}, $k\geq 3$ and $b'\geq v'\geq k^{2}-k+1$ (see Proposition~\ref{prfisher}), we get:
\begin{equation*}
m<\frac{\frac{(v'-1)^{2}}{k-1}+1-k}{v'-k}<\frac{(v'-1)^{2}}{(k-1)(v'-k)}\leq\frac{3}{2}\frac{v'-1}{k-1}<v'\leq b'
\end{equation*}
That $m<b'$ is equivalent to $n<b$ is evident by the counting argument for Corollary~\ref{conl}\ref{conl1}.

By \eqref{lewdmbeasy0}, having no minimal sub-BIBD intersection of size $0$ means that:
\begin{equation*}
m=\frac{\frac{(v'-1)(v'-k)}{k}}{\frac{(v-1)(v-v'^{2})}{v'(v'-1)}+\frac{v'(v'-1)}{k}}
\end{equation*}
If $v\geq v'^{2}$, this translates to $m\leq\frac{(v'-1)(v'-k)}{k}\cdot\frac{k}{v'(v'-1)}<1$, which is absurd, therefore we have $v<v'^{2}$; again by \eqref{mvvp} we have then $m<\frac{v'^{2}-k}{v'-k}$. We observe that $(v-1)(v-v'^{2})$ is negative and increasing when $v$ varies in the interval $\left(\frac{v'^{2}+1}{2},v'^{2}\right)$; $v'\left(v'-\frac{k}{2}\right)$ is in that interval, so for $v\in\left[v'\left(v'-\frac{k}{2}\right),v'^{2}\right)$ we have:
\begin{equation*}
m\leq\frac{\frac{(v'-1)(v'-k)}{k}}{\frac{\left(v'^{2}-\frac{kv'}{2}-1\right)\left(v'^{2}-\frac{kv'}{2}-v'^{2}\right)}{v'(v'-1)}+\frac{v'(v'-1)}{k}}=\frac{4(v'-1)^{2}(v'-k)}{4v'^{3}-2(k^{2}+4)v'^{2}+(k^{3}+4)v'+2k^{2}}
\end{equation*}
Since $k\geq 3$ and $v'\geq k^{2}-k+1$ we have $k^{2}+4\leq\frac{13}{7}v'$ and then $m<14$; this proves $m<b'$ except for the pairs $(v',k)=(7,3),(9,3),(13,4)$: these can be verified by hand (in all these cases $m<2$).

When $v\in\left(0,v'\left(v'-\frac{k}{2}\right)\right)$ we use again \eqref{mvvp}:
\begin{equation*}
m<\frac{v'^{2}-\frac{kv'}{2}-k}{v'-k}=v'+k\frac{\frac{v'}{2}-1}{v'-k}
\end{equation*}
Except for those pairs that we have already examined, we have $k\leq\frac{v'}{4}$ and then $m<v'+k\cdot\frac{2}{3}\frac{v'-2}{v'}<v'+k-1$; if we prove that when the sub-BIBD is not symmetric we have $b'\geq v'+k-1$ we are done. In a BIBD with $\lambda=1$ there are only two possible intersection sizes for pairs of blocks, $0$ and $1$: by Proposition~\ref{prfisher} if it is not symmetric it must have both of them, so in particular there exist two disjoint blocks $B_{1},B_{2}$; therefore, for any $x\in B_{1}$, there exist other $k$ blocks (one for each point of $B_{2}$) that intersect each other and $B_{1}$ all in $x$: this means that $v'\geq 1+(k-1)(k+1)=k^{2}$, hence by Lemma~\ref{lebibdeasy}\ref{lebibdeasy3} $b'=\frac{v'(v'-1)}{k(k-1)}\geq v'\left(1+\frac{1}{k}\right)\geq v'+k$.
\end{proof}

The BIBD defined in Examples~\ref{extheone3}-\ref{extheone4} have all sub-BIBDs pairwise intersecting in $k$ vertices, so they fall into the first case of Proposition~\ref{prspecm} (not in the second, because a $(7,3,1)$-BIBD and a $(13,4,1)$-BIBD are both symmetric).

\section{Concluding remarks}

We have investigated BIBDs that have well-distributed minimal sub-BIBDs, exploring what already existing results tell us in this particular case and describing many of their features, with respect to their size and their internal structure. However, many questions are very much open: for a start, Examples~\ref{extheone3}-\ref{extheone4} are the only designs satisfying Definition~\ref{dewdmb} explicitly known to the author. Others potentially emerge from the construction in Example~\ref{exconstr}; also, if we used an identical construction starting from a $(v,k,1)$-BIBD $V$ without any sub-BIBDs (i.e. the minimal sub-BIBD is $V$ itself) we could potentially create many more examples, provided that no new small subdesign is created by accident while pasting the various copies of $V$ as pointed out in that same example. In any case, it would be interesting to produce more examples of designs that satisfy the definition, especially since the BIBDs in Examples~\ref{extheone3}-\ref{extheone4} are rather special cases (they are symmetric designs of symmetric subdesigns); it would be also interesting to see whether there exist BIBDs with well-distributed minimal sub-BIBDs that are at the same time ``generic'', in the sense that all intersection sizes are possible, and ``primitive'', in the sense that they are not obtained by loosely (i.e. with intersection size $0$ or $1$) pasting together smaller examples: having a collection of such examples would give a better idea of what can or cannot be expected from this class of designs.

We have also investigated the relation between BIBDs with well-distributed minimal sub-BIBDs and the graph isomorphism problem. An important question that arose in that context was: can we expect that $n<b$? This fact is true in our explicit examples, and it is also true that if no new subdesigns are accidentally created then the construction in Example~\ref{exconstr} preserves this property (because $n<b\Longleftrightarrow m<b'$ and $m,b'$ remain the same after the pasting process); given the relevance of the question in the search for a faster algorithm to solve the GIP, it is of interest to prove that this is a general fact or conversely to describe examples that show the contrary and see what other characteristics they present.

A little heuristic reasoning might shed some light on the matter. Let us suppose that, for fixed $v,v',k,m$, we consider the entire class of $(v,v',m)$-BIBDs whose blocks pairwise intersect in $0,1,k$ vertices; by Lemma~\ref{leintblock} we know that in each block $D_{0}$ the $k$-size intersections form a $(v',k,m-1)$-BIBD (allowing repetition), and we have already noticed that in the case of BIBDs with well-distributed minimal sub-BIBDs this induced design is actually the union of $m-1$ copies of the same $(v',k,1)$-BIBD. As repetitions are allowed, if $S(v',k,1)$ is the number of BIBDs with those parameters then we also have $S(v',k,m-1)\geq\binom{S(v',k,1)+m-2}{m-1}$ (by taking the union of any $m-1$ of the BIBDs with $\lambda=1$); seeing the BIBDs as the result of pasting together $n$ different sub-BIBDs, the ratio of BIBDs satisfying the definition against the totality of BIBDs with the right parameters, judging by this rough estimate, appears then to be at most $S(v',k,1)^{n}\binom{S(v',k,1)+m-2}{m-1}^{-n}$. Results like \cite{Wi74} suggest that the number of $(v',k,1)$-BIBDs should be of the form $e^{O(v'^{2}\log v')}$, or in any case a large exponential (recent developments \cite{Ke15} \cite{GKLO17} \cite{Ke18} seem to indicate that we could be able to know the answer in the near future); comparing that ratio with $S(v,v',m)$, which would be the number of possible BIBDs with those parameters, we would get that in order to have the possibility of a BIBD satisfying Definition~\ref{dewdmb}:
\begin{equation*}
v^{2}\log v\gtrsim (v'^{2}\log v')(m-2)n\approx m^{2}v^{2}\log v' \ \Longrightarrow \ v\gtrsim v'^{m^{2}}
\end{equation*}
In the case $m\geq b'=\frac{v'(v'-1)}{k(k-1)}$, this would make us expect $v$ to be huge with respect to $v'$.

\section*{Acknowledgements}

The author thanks H. A. Helfgott for introducing him to the graph isomorphism problem and for discussions about his paper \cite{He17} on the subject.

\end{document}